\documentclass[11pt,english]{amsart}

\usepackage[T1]{fontenc}
\usepackage[latin9]{inputenc}
\usepackage{geometry}
\usepackage{color}
\usepackage{textcomp}
\usepackage{amsthm}
\usepackage{amstext}
\usepackage{amssymb}

\makeatletter
\numberwithin{equation}{section}
\numberwithin{figure}{section}
\theoremstyle{plain}
\newtheorem{thm}{\protect\theoremname}
  \theoremstyle{plain}
  \newtheorem{lem}[thm]{\protect\lemmaname}
  \theoremstyle{remark}
  \newtheorem{rem}[thm]{\protect\remarkname}
  \theoremstyle{plain}
  \newtheorem{prop}[thm]{\protect\propositionname}
  \theoremstyle{definition}
  \newtheorem{defn}[thm]{\protect\definitionname}
  \theoremstyle{definition}
  \newtheorem{example}[thm]{\protect\examplename}
  \theoremstyle{plain}
  \newtheorem{cor}[thm]{\protect\corollaryname}

\usepackage{tikz}

\@ifundefined{definecolor}
 {\usepackage{color}}{}

\newcommand{\ad}[2]{
\begin{tikzpicture}
\foreach \p/\t/\w in {#1}
 {
  \node [below] at (\p,0) {$\scriptstyle{\w}$};
  \node at (\p,0) {\t};
 }
\foreach \s/\r/\sa/\ea in {#2}
 {
  \draw (\s,0.2) arc [radius=\r, start angle=\sa, end angle=\ea];
 }
\end{tikzpicture}}

\newcommand{\cir}{$\bullet$}
\newcommand{\x}{$\times$}
\date{\today}

\makeatother

\usepackage{babel}
  \providecommand{\corollaryname}{Corollary}
  \providecommand{\definitionname}{Definition}
  \providecommand{\examplename}{Example}
  \providecommand{\lemmaname}{Lemma}
  \providecommand{\propositionname}{Proposition}
  \providecommand{\remarkname}{Remark}
\providecommand{\theoremname}{Theorem}

\begin{document}

\title{Kac-Wakimoto character formula for the general linear Lie superalgebra}

\author{Michael Chmutov, Crystal Hoyt, Shifra Reif}

\thanks{The second author was supported in part by an Aly Kaufman Fellowship
at the Technion. The third author was partially supported by NSF RTG
grant DMS 0943832.}
\begin{abstract}
Character formulas for Lie superalgebras have been shown to have important
applications to number theory and combinatorics. We prove the Kac-Wakimoto
character formula for the general linear Lie superalgebra $\mathfrak{gl}\left(m|n\right)$.
This formula specializes to the well-known Kac-Weyl character formula
when the modules are typical and to the Weyl denominator identity
when the module is trivial. We also prove a determinantal character
formula for KW-modules using the Kac-Wakimoto character formula.
\end{abstract}
\maketitle
 \date{\currenttime} \maketitle

\section{Introduction}

It is well known that character theory for Lie superalgebras is a
non-trivial generalization of the classical theory. The search for
a Kac-Weyl type character formula has been a driving force in the
field. The heart of the problem lies in the existence of the so called
``atypical roots''. In their absence, a formula similar to the classical
Weyl character formula was proven for finite dimensional typical highest
weight modules by Kac in 1977 \cite{K1,K2,K3}. For the singly atypical
weights (those with only one atypical root), a closed formula was
proven by Bernstein and Leites for $\mathfrak{gl}\left(m|n\right)$
in 1980 using geometrical methods \cite{BL}. It was a long standing
question to generalize this formula and many people contributed, including
Van der Jeugt, Hughes, King and Thierry-Mieg who proposed a conjectural
character formula for $\mathfrak{gl}(m|n)$ in 1990 \cite{VHKT},
which was later proven by Su and Zhang in 2007 \cite{SZ1}.

In 1996, Serganova introduced the generalized Kazhdan-Lusztig polynomials
for $\mathfrak{gl}(m|n)$ and used them to give a general character
formula for finite dimensional irreducible representations of $\mathfrak{gl}\left(m|n\right)$
\cite{S1,S2}. For each $\lambda$ and $\mu$ dominant integral weights,
evaluating the Kazhdan-Lusztig polynomial $K_{\lambda,\mu}\left(q\right)$
at $q=-1$ was shown to yield the multiplicity of the Kac module $\overline{L}\left(\mu\right)$
inside the finite dimensional simple module $L\left(\lambda\right)$.
In 2003, Brundan gave an algorithm for computing the generalized Kazhdan-Lusztig
polynomials for $\mathfrak{gl}(m|n)$ in terms of weight diagrams
\cite{B}. 

In 2007, Su and Zhang used Brundan's algorithm to compute the generalized
Kazhdan-Lusztig polynomials and to prove a Kac-Weyl type character
formula for finite dimensional simple modules of $\mathfrak{gl}(m|n)$
in the standard choice of simple roots, which we refer to as the Su-Zhang
character formula \cite[4.43]{SZ1}. When the highest weight $\lambda$
is ``totally connected'' (see Definition \ref{def: totally connected}),
every non-zero Kazhdan-Lusztig polynomial is a monomial with coefficient
$1$, which drastically simplifies the Su-Zhang character formula
(\ref{eq: SZ arrow formula}) \cite[4.46]{SZ1}. 

We use the Su-Zhang character formula to prove the Kac-Wakimoto character
formula for $\mathfrak{gl}(m|n)$, which was stated by Kac and Wakimoto
in 1994 \cite{KW}.
\begin{thm}
\label{thm:main}Let $L$ be a finite dimensional simple module. For
any choice of simple roots $\pi$ and weight $\lambda,$ such that
$L=L_{\pi}(\lambda)$ and $\pi$ contains a $\lambda^{\rho}$-maximal
isotropic subset $S_{\lambda}$ we have

\begin{equation}
e^{\rho}R\cdot\mbox{ch }L_{\pi}\left(\lambda\right)=\frac{1}{r!}\sum_{w\in W}(-1)^{l(w)}w\left(\frac{e^{\lambda^{\rho}}}{\prod_{\beta\in S_{\lambda}}\left(1+e^{-\beta}\right)}\right)\label{eq:KW-1}
\end{equation}

where $r=|S_{\lambda}|$.
\end{thm}
We call a finite dimensional simple module $L$ \emph{a KW-module
}if there exists a set of simple roots $\pi$ that satisfies the hypothesis
of Theorem \ref{thm:main}, and we call such $\pi$ an \emph{admissible}
choice of simple roots for $L$. 

Our proof goes as follows. We prove that a finite dimensional simple
module is a KW-module if and only if its highest weight with respect
to the standard choice of simple roots is totally connected (Theorem
\ref{thm:tame equivalent totally connected}), by presenting an algorithm
to move between these different bases. We generalize the arc diagrams
defined in \cite{GKMP}, and we use these diagrams to define the steps
of our algorithm, which are composed of a specified sequence of odd
reflections. We show that each step of the algorithm preserves a generalized
version of the Su-Zhang character formula, and that this generalized
formula specializes to the Kac-Wakimoto character formula when the
set of simple roots is admissible, thus proving that KW-modules for
$\mathfrak{gl}(m|n)$ are tame.

In this way, we obtain a character formula for each admissible set
of simple roots $\pi$ and $\lambda^{\rho}$-maximal isotropic subset
$S_{\lambda}\subset\pi$. We also obtain character formulas for the
sets of simple roots encountered when applying the ``shortening algorithm''
to a totally connected weight, but not for all sets of simple roots.
When the module is trivial, these character formulas specialize to
the denominator identities obtained by Gorelik, Kac, Möseneder Frajria
and Paolo-Papi in \cite{G,GKMP}.

We use the Kac-Wakimoto character formula to prove a determinantal
character formula for KW-modules (Theorem \ref{thm:determinantal character formula}),
which is motivated by the determinantal character formula proven in
\cite{MV2} for critical modules labeled by non-intersecting composite
partitions. Our determinantal character formula can be expressed using
the data of the ``special arc diagram'' for a KW-module $L$ (see
Definition \ref{def: special tame}), and is useful for computer computations.

To make the paper self-contained, we give a simplified version of
the proof of the Su-Zhang character formula in the special case that
the highest weight of the module is totally connected, i.e., a KW-module.
Along the way, we obtain another characterization of KW-modules in
terms of their Kazhdan-Lusztig polynomials (see Corollary \ref{cor:kl polys for tame}),
and we obtain the character formula of Theorem \ref{thm:other character formula in standard},
which is motivated by the denominator identity given in \cite[(1.10)]{GKMP}
for the standard choice of simple roots.

An important class of KW-modules are covariant modules. The tensor
module $T(V)$ of the $(m+n)$-dimensional natural representation
$V$ of $\mathfrak{gl}(m|n)$ is completely reducible, and its irreducible
components are called\emph{ covariant modules}. These modules and
their characters (super-Schur functions) were studied by Berele and
Regev \cite{BR}, and Sergeev \cite{S}, who in particular gave a
necessary and sufficient condition on a weight $\lambda$ in terms
of Young diagrams to be the highest weight of a covariant module for
$\mathfrak{gl}(m|n)$ (with respect to the standard set of simple
roots). Specifically, the irreducible components of the $k$-th tensor
power of $V$ are parametrized by Young diagrams of size $k$ contained
in the $\left(m,n\right)$-hook. Using this description, it is easy
to show that the highest weight of a covariant module is totally connected
(see Example \ref{Example:covariant modules have t.c. highest weight}).
Hence, it follows from Theorem \ref{thm:tame equivalent totally connected}
that covariant modules are KW-modules, which was originally proven
in \cite{MV1}.

Acknowledgments. We would like to thank the following people for useful
and interesting discussions: Jonah Blasiak, Ivan Dimitrov, Sergey
Fomin, Maria Gorelik, Thomas Lam and Ricky Liu.

\section{Preliminaries}

\subsection{The general linear Lie superalgebra.}

In this paper, $\mathfrak{g}$ will always denote the general linear
Lie superalgebra $\mathfrak{gl}(m|n)$ over the complex field $\mathbb{C}$.
As a vector space, $\mathfrak{g}$ can be identified with the endomorphism
algebra $\mbox{End}(V_{\bar{0}}\oplus V_{\bar{1}})$ of a $\mathbb{Z}_{2}$-graded
vector space $V_{\bar{0}}\oplus V_{\bar{1}}$ with $\mbox{dim }V_{\bar{0}}=m$
and $\mbox{dim }V_{\bar{1}}=n$. Then $\mathfrak{g}=\mathfrak{g}_{\bar{0}}\oplus\mathfrak{g}_{\bar{1}}$,
where 
\[
\mathfrak{g}_{\bar{0}}=\mbox{End}(V_{\bar{0}})\oplus\mbox{End}(V_{\bar{1}})\quad\mathrm{and}\quad\mathfrak{g}_{\bar{1}}=\mbox{Hom}(V_{\bar{0}},V_{\bar{1}})\oplus\mbox{Hom}(V_{\bar{1}},V_{\bar{0}}).
\]
A homogeneous element $x\in\mathfrak{g}_{\bar{0}}$ has degree $0$,
denoted $\mbox{deg}(x)=0$, while $x\in\mathfrak{g}_{\bar{1}}$ has
degree $1$, denoted $\mbox{deg}(x)=1$. We define a bilinear operation
on $\mathfrak{g}$ by letting 
\[
[x,y]=xy-(-1)^{\mathrm{deg}(x)\mathrm{deg}(y)}yx
\]
on homogeneous elements and then extending linearly to all of $\mathfrak{g}$.

By fixing a basis of $V_{\bar{0}}$ and $V_{\bar{1}}$, we can realize
$\mathfrak{g}$ as the set of $(m+n)\times(m+n)$ matrices, where
\[
\mathfrak{g}_{\bar{0}}=\left\{ \left(\begin{array}{cc}
A & 0\\
0 & B
\end{array}\right)\mid A\in M_{m,m},\ B\in M_{n,n}\right\} \text{ and }\mathfrak{g}_{\bar{1}}=\left\{ \left(\begin{array}{cc}
0 & C\\
D & 0
\end{array}\right)\mid C\in M_{m,n},\ D\in M_{n,m}\right\} ,
\]
and $M_{r,s}$ denotes the set of $r\times s$ matrices.

The Cartan subalgebra $\mathfrak{h}$ of $\mathfrak{g}$ is the set
of diagonal matrices, and it has a natural basis 
\[
\{E_{1,1},\ldots,E_{m,m};E_{m+1,m+1},\ldots,E_{m+n,m+n}\},
\]
where $E_{ij}$ denotes the matrix whose $ij$-entry is $1$ and there
are $0$'s elsewhere. Fix the dual basis $\{\varepsilon_{1},\ldots,\varepsilon_{m};\delta_{1},\ldots,\delta_{n}\}$
for $\mathfrak{h}^{*}$. We define a bilinear form on $\mathfrak{h}^{*}$
by $(\varepsilon_{i},\varepsilon_{j})=\delta_{ij}=-(\delta_{i},\delta_{j})$
and $(\varepsilon_{i},\delta_{j})=0$, and use it to identify $\mathfrak{h}$
with $\mathfrak{h}^{*}$.

Then $\mathfrak{g}$ has a root space decomposition $\mathfrak{g}=\mathfrak{h}\oplus\left(\bigoplus_{\alpha\in\Delta_{\bar{0}}}\mathfrak{g}_{\alpha}\right)\oplus\left(\bigoplus_{\alpha\in\Delta_{\bar{1}}}\mathfrak{g}_{\alpha}\right),$
where the set of roots of $\mathfrak{g}$ is $\Delta=\Delta_{\bar{0}}\cup\Delta_{\bar{1}}$,
with 
\[
\begin{aligned} & \Delta_{\bar{0}}=\{\varepsilon_{i}-\varepsilon_{j}\mid1\leq i\neq j\leq m\}\cup\{\delta_{k}-\delta_{l}\mid1\leq k\neq l\leq n\},\\
 & \Delta_{\bar{1}}=\{\pm(\varepsilon_{i}-\delta_{k})\mid1\leq i\leq m,\ 1\leq k\leq n\},
\end{aligned}
\]
and $\mathfrak{g}_{\varepsilon_{i}-\varepsilon_{j}}=\mathbb{C}E_{ij},\ \mathfrak{g}_{\delta_{k}-\delta_{l}}=\mathbb{C}E_{m+k,m+l},\ \mathfrak{g}_{\varepsilon_{i}-\delta_{k}}=\mathbb{C}E_{i,m+k},\ \mathfrak{g}_{\delta_{k}-\varepsilon_{i}}=\mathbb{C}E_{m+k,i}.$

A set of simple roots $\pi\subset\Delta$ determines a decomposition
of $\Delta$ into positive and negative roots, $\Delta=\Delta^{+}\cup\Delta^{-}$.
There is a corresponding triangular decomposition of $\mathfrak{g}$
given by $\mathfrak{g}=\mathfrak{n}^{+}\oplus\mathfrak{h}\oplus\mathfrak{n}^{-}$,
where $\mathfrak{n}^{\pm}=\oplus_{\alpha\in\Delta^{\pm}}\mathfrak{g}_{\alpha}$.
Let $\Delta_{\bar{d}}^{+}=\Delta_{\bar{d}}\cap\Delta^{+}$ for $d\in\{0,1\}$,
and define $\rho_{\pi}=\frac{1}{2}\sum_{\alpha\in\Delta_{\bar{0}}^{+}}\alpha-\frac{1}{2}\sum_{\alpha\in\Delta_{\bar{1}}^{+}}\alpha.$
Then for $\alpha\in\pi$, we have $(\rho_{\pi},\alpha)=(\alpha,\alpha)/2$.

The Weyl group of $\mathfrak{g}$ is $W=Sym_{m}\times Sym_{n}$, and
$W$ acts on $\mathfrak{h}^{*}$ by permuting the indices of the $\varepsilon$'s
and by permuting the indices of the $\delta$'s. In particular, the
even reflection $s_{\varepsilon_{i}-\varepsilon_{j}}$ interchanges
the $i$ and $j$ indices of the $\varepsilon$'s and fixes all other
indices, while $s_{\delta_{k}-\delta_{l}}$ interchanges the $k$
and $l$ indices of the $\delta$'s and fixes all other indices.

A proof of the following lemma can be found in \cite[4.1.1]{G}.
\begin{lem}
\label{lem:stabilizer has a reflection}For any $\mu\in\mathfrak{h}_{\mathbb{R}}^{*}$
the stabilizer of $\mu$ in $W$ is either trivial or contains a reflection.
\end{lem}
Suppose $\beta\in\pi$ is an odd (isotropic) root. An odd reflection
$r_{\beta}$ of each $\alpha\in\pi$ is defined as follows. 
\[
r_{\beta}(\alpha)=\left\{ \begin{array}{cl}
-\alpha, & \text{if }\beta=\alpha\\
\alpha, & \text{if }(\alpha,\beta)=0\\
\alpha+\beta, & \text{if }(\alpha,\beta)\neq0
\end{array}\right.
\]
Then $r_{\beta}\pi:=\{r_{\beta}(\alpha)\mid\alpha\in\pi\}$ is also
a set of simple roots for $\mathfrak{g}$ \cite{S3}. The corresponding
root decomposition is $\Delta=\Delta'^{+}\cup\Delta'^{-}$, where
$\Delta'{}_{\bar{0}}^{+}=\Delta_{\bar{0}}^{+}$ is unchanged, and
$\Delta'{}_{\bar{1}}^{+}=\left(\Delta_{\bar{1}}^{+}\setminus\{\beta\}\right)\cup\{-\beta\}$.
Using a sequence of even and odd reflections, one can move between
any two sets of simple roots for $\mathfrak{g}$. Moreover, if $\pi$
and $\pi''$ have the property that $\Delta''{}_{\bar{0}}^{+}=\Delta_{\bar{0}}^{+}$,
then there exists a sequence of odd reflections from $\pi$ to $\pi''$.

We denote by $\pi_{st}$ the standard choice of simple roots. 
\[
\pi_{st}=\left\{ \varepsilon_{1}-\varepsilon_{2},\ldots,\varepsilon_{m-1}-\varepsilon_{m},\varepsilon_{m}-\delta_{1},\delta_{1}-\delta_{2}\ldots,\delta_{n-1}-\delta_{n}\right\} 
\]
The corresponding decomposition $\Delta=\Delta^{+}\cup\Delta^{-}$
is given by 
\begin{equation}
\Delta_{\bar{0}}^{+}=\{\varepsilon_{i}-\varepsilon_{j}\}_{1\leq i<j\leq m}\cup\{\delta_{k}-\delta_{l}\}_{1\leq k<l\leq n}\quad\text{ and}\quad\Delta_{\bar{1}}^{+}=\{\varepsilon_{i}-\delta_{k}\}_{1\leq i\leq m,\ 1\leq k\leq n}.\label{eq:set of positive roots}
\end{equation}
The standard choice $\pi_{st}$ has the unique property that $W$
fixes $\Delta_{\bar{1}}^{+}$. Moreover, $\pi_{st}$ contains a basis
for $\Delta_{\bar{0}}^{+}$, which we denote by $\pi_{\bar{0}}$.

The root lattice $Q=\sum_{\alpha\in\pi}\mathbb{Z}\alpha$ is independent
of the choice of $\pi$.  Let $Q_{\pi}^{+}=\sum_{\alpha\in\pi}\mathbb{N}\alpha$,
where $\mathbb{N}=\{0,1,2,\ldots\}$, and define a partial order on
$\mathfrak{h}^{*}$ by $\mu>\nu$ when $\mu-\nu\in Q_{st}^{+}$.
\begin{rem}
For convenience, we fix $\Delta_{\bar{0}}^{+}$ as in (\ref{eq:set of positive roots}).
This choice is arbitrary since we can relabel the indices of the $\varepsilon$'s
and $\delta$'s. We let $l(w)$ denote the length of $w\in W$ with
respect to the set of simple reflections $s_{\varepsilon_{1}-\varepsilon_{2}},\dots,s_{\varepsilon_{m-1}-\varepsilon_{m}};s_{\delta_{1}-\delta_{2}},\dots,s_{\delta_{n-1}-\delta_{n}}$
which generate $W$.
\end{rem}

\subsection{Finite dimensional modules for $\mathfrak{\mathfrak{gl(m|n)}}$}

For each set of simple roots $\pi$ and weight $\lambda\in\mathfrak{h}^{*}$,
the \emph{Verma module} of highest weight $\lambda$ is the induced
module 
\[
M_{\pi}(\lambda):=\mbox{Ind}_{\mathfrak{n}^{+}\oplus\mathfrak{h}}^{\mathfrak{g}}\mathbb{C}_{\lambda},
\]
where $\mathbb{C}_{\lambda}$ is the one-dimensional module such that
$h\in\mathfrak{h}$ acts by scalar multiplication of $\lambda(h)$
and $\mathfrak{n}^{+}$ acts trivially. The Verma module $M_{\pi}(\lambda)$
has a unique simple quotient, which we denote by $L_{\pi}(\lambda)$
or simply by $L(\lambda)$. Given $\pi$ and $\lambda$, we denote
by $\lambda^{\rho}$ or $\lambda_{\pi}^{\rho}$ 
\[
\lambda^{\rho}:=\lambda+\rho_{\pi}.
\]
If $L_{\pi}\left(\lambda\right)\cong L_{\pi'}\left(\lambda'\right)$
for some $\lambda,\lambda'\in\mathfrak{h}^{*}$ and $\pi'=r_{\beta}\pi$
for an odd reflection $r_{\beta}$ with $\beta\in\pi$, then 
\[
(\lambda')^{\rho}=\left\{ \begin{array}{cl}
\lambda^{\rho}, & \text{if }\left(\lambda,\beta\right)\ne0\\
\lambda^{\rho}+\beta, & \text{if }\left(\lambda,\beta\right)=0.
\end{array}\right.
\]

For each $\lambda\in\mathfrak{h}^{*}$, let $L_{\bar{0}}(\lambda)$
denote the simple highest weight $\mathfrak{g}_{\bar{0}}$-module
with respect to $\pi_{\bar{0}}$. The {\em Kac module} of highest
weight $\lambda$ with respect to $\pi_{st}$ is the induced module
\[
\overline{L}(\lambda):=\mbox{Ind}_{\mathfrak{g}_{\bar{0}}\oplus\mathfrak{n}_{\bar{1}}^{+}}^{\mathfrak{g}}L_{\bar{0}}(\lambda)
\]
defined by letting $\mathfrak{n}_{\bar{1}}^{+}:=\oplus_{\alpha\in\Delta_{\bar{1}}^{+}}\mathfrak{g}_{\alpha}$
act trivially on the $\mathfrak{g}_{\bar{0}}$-module $L_{\bar{0}}(\lambda)$.
The unique simple quotient of $\overline{L}(\lambda)$ is $L_{\pi_{st}}(\lambda)$.

A weight $\lambda\in\mathfrak{h}^{*}$ is called \textit{dominant}
if $\frac{2(\lambda,\alpha)}{(\alpha,\alpha)}\geq0$ for all $\alpha\in\Delta_{\bar{0}}^{+}$,
\textit{strictly dominant} if $\frac{2(\lambda,\alpha)}{(\alpha,\alpha)}>0$
for all $\alpha\in\Delta_{\bar{0}}^{+}$ and \textit{integral} if
$\frac{2(\lambda,\alpha)}{(\alpha,\alpha)}\in\mathbb{Z}$ for all
$\alpha\in\Delta_{\bar{0}}^{+}$. It is sufficient to check dominance
and integrality on the set of simple roots $\pi_{\bar{0}}\subset\Delta_{\bar{0}}^{+}$. 

The proof of following lemma is straight-forward when viewing $L_{\pi}(\lambda)$
as a $\mathfrak{g}_{\bar{0}}$-module.
\begin{lem}
If the simple module $L_{\pi}(\lambda)$ is finite dimensional, then
$\lambda$ is a dominant integral weight.
\end{lem}
For a proof of the following proposition see for example \cite[14.1.1]{M}.
\begin{prop}
\label{sub:rho^st strictly dominant}For $\mathfrak{g}=\mathfrak{gl}(m|n)$
and $\lambda\in\mathfrak{h}^{*}$, the following are equivalent: \end{prop}
\begin{enumerate}
\item the simple highest weight $\mathfrak{g}$-module $L_{\pi_{st}}(\lambda)$
is finite dimensional; 
\item the Kac module $\overline{L}(\lambda)$ is finite dimensional; 
\item the simple highest weight $\mathfrak{g}_{\bar{0}}$-module $L_{\bar{0}}(\lambda)$
is finite dimensional; 
\item $\lambda$ is a dominant integral weight; 
\item $\lambda_{st}^{\rho}$ is a strictly dominant integral weight. 
\end{enumerate}

\subsection{Atypical modules.\label{sub:Atypical-modules.}}

The \textit{atypicality} of $L_{\pi}(\lambda)$ is the maximal number
of linearly independent roots $\beta_{1},...,\beta_{r}$ such that
$\left(\beta_{i},\beta_{j}\right)=0$ and $\left(\lambda^{\rho},\beta_{i}\right)=0$
for $i,j=1,\ldots,r$. Such a set $S=\left\{ \beta_{1},\ldots,\beta_{r}\right\} $
is called a \textit{$\lambda^{\rho}$-maximal isotropic set}. The
module $L_{\pi}(\lambda)$ is called \textit{typical} if this set
is empty, and \textit{atypical} otherwise. The atypicality of a simple
finite dimensional module is independent of the choice of simple roots.
\begin{defn}
We call a simple finite dimensional module $L$ \textit{\emph{a }}\textit{KW-module}
if $L\cong L_{\pi}(\lambda)$ for some choice of simple roots $\pi$
which contains a $\lambda^{\rho}$-maximal isotropic set $S\subset\pi$.
In this case, we call such $\pi$ an \emph{admissible}\textit{ }\textit{\emph{choice
of simple roots}}\emph{ }for $L$.
\end{defn}
Let $P$ denote the set of integral weights, $P^{+}$ the set of dominant
integral weights, and define 
\[
\mathbb{P}^{+}=\{\mu\in P^{+}\mid(\mu_{\pi}^{\rho},\varepsilon_{i})\in\mathbb{Z},\ (\mu_{\pi}^{\rho},\delta_{j})\in\mathbb{Z}\}.
\]
Note that the definition of $\mathbb{P}^{+}$ is independent of the
choice of $\pi$, since changing the set of simple roots by an odd
reflection, only changes the entries of $\lambda_{\pi}^{\rho}$ by
integer values.
\begin{rem}
\label{all integers}When studying the character of a simple finite
dimensional atypical module, we may restrict to the case that $\lambda\in\mathbb{P}^{+}$.
Indeed, let $\lambda\in P^{+}$, then the module $L_{\pi}(\lambda)$
is atypical if and only if $(\lambda_{\pi}^{\rho},\varepsilon_{i})=(\lambda_{\pi}^{\rho},\delta_{j})$
for some $(\varepsilon_{i}-\delta_{j})\in\Delta_{\bar{1}}$. So, by
tensoring $L_{\pi}(\lambda)$ with a $1$-dimensional module with
character $e^{c\left(\sum_{i=1}^{m}\varepsilon_{i}-\sum_{j=1}^{n}\delta_{j}\right)}$
for appropriate $c\in\mathbb{C}$, we obtain a module $L_{\pi}(\lambda')$
with $\lambda'\in\mathbb{P}^{+}$.
\end{rem}
Fix a set of simple roots $\pi$, a weight $\lambda\in\mathbb{P}^{+}$
and a $\lambda^{\rho}$-maximal isotropic set $S_{\lambda}\subset\Delta_{\bar{1}}^{+}$.
Write 
\begin{equation}
\lambda_{\pi}^{\rho}=\sum_{i=1}^{m}a_{i}\varepsilon_{i}-\sum_{j=1}^{n}b_{j}\delta_{j}.\label{eq:lambda rho expansion}
\end{equation}
We refer to the coefficient $a_{i}$ (resp. $b_{j}$) as the\textit{
$\varepsilon_{i}$-entry} (resp. $\delta_{j}$-entry). If $\pm(\varepsilon_{k}-\delta_{l})\in S_{\lambda}$,
then we call the $\varepsilon_{k}$ and \textrm{$\delta_{l}$} entries
\emph{atypical}. Otherwise, an entry is called \textit{typical}. 

We denote by $(\lambda_{\pi}^{\rho})^{\Uparrow}$ the element obtained
from $\lambda_{\pi}^{\rho}$ by replacing all its atypical entries
by the maximal atypical entry. Note that this can depend on choice
of $S_{\lambda}$ and not only on $\lambda_{\pi}^{\rho}$. However,
for $\pi_{st}$ and $\lambda\in\mathbb{P}^{+}$, there is a unique
$S_{\lambda}\subset\Delta^{+}$. If $\nu\in\mathfrak{h}^{*}$ can
be written as $\nu=\sum_{\alpha\in S_{\lambda}}k_{\alpha}\alpha$,
then we define 
\[
\left|\nu\right|_{S_{\lambda}}:=\sum_{\alpha\in S_{\lambda}}k_{\alpha}.
\]
Observe that $|(\lambda_{\pi}^{\rho})^{\Uparrow}-\lambda_{\pi}^{\rho}|_{S_{\lambda}}$
is non-negative integer.

\subsection{Arc diagrams}

We generalize the arc diagrams defined in \cite{GKMP}. Let $L$ be
a finite dimensional atypical module. For each set of simple roots
$\pi$, weight $\lambda\in\mathbb{\mathbb{P}}^{+}$ such that $L=L_{\pi}\left(\lambda\right)$
and $\lambda^{\rho}$-maximal isotropic set $S_{\lambda}\subset\Delta_{\bar{1}}^{+}$,
there is an arc diagram that encodes the data: $\pi$, $\lambda_{\pi}^{\rho}$
and $S_{\lambda}$.

In order to define the arc diagram corresponding to the data $(\pi,\lambda_{\pi}^{\rho},S_{\lambda})$,
we first define a total order on the set $\{\varepsilon_{1},\ldots,\varepsilon_{m}\}\cup\{\delta_{1},\ldots,\delta_{n}\}$
determined by $\Delta^{+}$. In particular, $\varepsilon_{i}<\varepsilon_{i+1}$,
$\delta_{j}<\delta_{j+1}$ and for each $i$ and $j$ we let
\begin{alignat*}{1}
\delta_{j}<\varepsilon_{i}, & \text{\hspace{1em}if }(\delta_{j}-\varepsilon_{i})\in\Delta^{+},\\
\varepsilon_{i}<\delta_{j}, & \text{\hspace{1em}if }(\varepsilon_{i}-\delta_{j})\in\Delta^{+}.
\end{alignat*}
Let $T=\{\gamma_{1},\ldots,\gamma_{m+n}\}$ be this totally ordered
set, and express $\lambda_{\pi}^{\rho}$ as in (\ref{eq:lambda rho expansion}).
The \emph{nodes} and \emph{entries} of the diagram are determined
from left to right by $\gamma_{k}\in T$, $k=1,\ldots,m+n$, by putting
the node $\bullet$ labeled with the entry $a_{i}$ if $\gamma_{k}=\varepsilon_{i}$,
and by putting the node $\times$ labeled with the entry $b_{j}$
if $\gamma_{k}=\delta_{j}$. The set $S_{\lambda}$ determines an
\emph{arc arrangement} as follows. The $\varepsilon_{i}-$node $\bullet$
and the $\delta_{j}$-node $\times$ are connected by an arc when
$\pm(\varepsilon_{i}-\delta_{j})\in S_{\lambda}$, and in this case
$a_{i}=b_{j}$. A node and its entry are called \emph{atypical} if
the node is connected by some arc, and \emph{typical} otherwise.
\begin{example}
Suppose 
\begin{eqnarray*}
\pi & = & \left\{ \varepsilon_{1}-\varepsilon_{2},\varepsilon_{2}-\delta_{1},\delta_{1}-\varepsilon_{3},\varepsilon_{3}-\delta_{2},\delta_{2}-\delta_{3},\delta_{3}-\varepsilon_{4}\right\} \\
\lambda_{\pi}^{\rho} & = & 7\varepsilon_{1}+5\varepsilon_{2}+5\varepsilon_{3}+2\varepsilon_{4}-5\delta_{1}-6\delta_{2}-7\delta_{3}\\
S_{\lambda} & = & \left\{ \varepsilon_{1}-\delta_{3},\delta_{1}-\varepsilon_{3}\right\} .
\end{eqnarray*}
Then the corresponding arc diagram will be: 
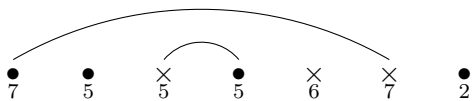
\begin{equation*}
\begin{tikzpicture} 
%diagram
\node at (0,0) {$\bullet$};
\node [below] at (0,0) {$\scriptstyle 7$};
\node at (1,0) {$\bullet$};
\node [below] at (1,0) {$\scriptstyle 5$};
\node at (2,0) {$\times$};
\node [below] at (2,0) {$\scriptstyle 5$};
\node at (3,0) {$\bullet$};
\node [below] at (3,0) {$\scriptstyle 5$};
\node at (4,0) {$\times$};
\node [below] at (4,0) {$\scriptstyle 6$};
\node at (5,0) {$\times$};
\node [below] at (5,0) {$\scriptstyle 7$};
\node at (6,0) {$\bullet$};
\node [below] at (6,0) {$\scriptstyle 2$};

%arcs
\draw (3,0.2) arc [radius=0.65, start angle=40, end angle= 140];
\draw (5,0.2) arc [radius=5, start angle=60, end angle= 120];

\end{tikzpicture}
\end{equation*}
\end{example}
We can recover the data $(\pi,\lambda_{\pi}^{\rho},S_{\lambda})$
from an arc diagram as follows. Label the $\bullet$-nodes from left
to right by $\varepsilon_{1},\ldots,\varepsilon_{m}$, and the $\times$-nodes
by $\delta_{1},\ldots,\delta_{n}$. Let $T=\{\gamma_{1},\ldots,\gamma_{m+n}\}$
be the ordered set determined by this labeling. Then $\pi=\{\gamma_{1}-\gamma_{2},\gamma_{2}-\gamma_{3},\ldots,\gamma_{m+n-1}-\gamma_{m+n}\}$,
$\lambda_{\pi}^{\rho}$ is given by (\ref{eq:lambda rho expansion}),
where $a_{i}$ is the $\varepsilon_{i}-$entry and $b_{j}$ is the
$\delta_{j}$-entry, and $S_{\lambda}=\{\gamma_{i}-\gamma_{j}\mid i<j\text{ and }\gamma_{i}\text{ is connected by an arc to }\gamma_{j}\}$. 
\begin{rem}
\label{rem: arc diagram entries}All entries of an arc diagram are
integers, and adjacent $\bullet$ entries are strictly decreasing,
while adjacent $\times$ entries are strictly increasing, because
$\lambda\in\mathbb{P}^{+}$ and $\lambda_{\pi}^{\rho}=\lambda+\rho_{\pi}$
(see Remark \ref{all integers}).
\end{rem}

\begin{rem}
We call the arc diagram for the standard choice of simple roots the
\emph{standard arc diagram}. For the standard arc diagram, there is
only one possible arc arrangement and all arcs are ``nested''. See
for example Diagram (\ref{diagram example in standard choice}).
\end{rem}

Since $S_{\lambda}\subset\Delta_{\bar{1}}^{+}$, an arc is always
connected to a $\bullet$ -node and a $\times$-node. We call an arc
\emph{short} if the $\bullet$ -node and $\times$-node are adjacent.
We say that an arc has $\bullet-\times$ type if the $\bullet$ precedes
the $\times$, and $\times-\bullet$ type if the $\times$ precedes
the $\bullet$. Note that no two arcs can share an endpoint, since
$S_{\lambda}$ is a $\lambda^{\rho}$-maximal isotropic set. Moreover,
an arc diagram by definition has a maximal arc arrangement, that is,
it is not possible to add an arc to the diagram between typical $\bullet$
and $\times$ nodes with equal entries.

Adjacent $\bullet$ and $\times$ nodes correspond to an odd simple
root $\beta$, and applying the odd reflection $r_{\beta}$ swaps
these nodes. In terms of the diagram it means that if $a\ne b$, then\begin{equation}\label{odd reflection of diagrams}
\begin{aligned}
&\begin{tikzpicture} 
%diagram
\node at (0,0) {$\bullet$};
\node  at (-1,0) {$\dots$};
\node [below] at (0,0) {$\scriptstyle a$};
\node at (1,0) {$\times$};
\node [below] at (1,0) {$\scriptstyle b$};
\node at (2,0) {$\dots$};
\node at (4,0) {$\stackrel{r_{\beta}}{\longleftrightarrow}$};
\node at (6,0) {$\dots$};

\node at (7,0) {$\times$};
\node [below] at (7,0) {$\scriptstyle b$};
\node at (8,0) {$\bullet$};
\node [below] at (8,0) {$\scriptstyle a$};
\node  at (9,0) {$\dots$};
\end{tikzpicture}\\
&\begin{tikzpicture} 
%diagram
\node at (0,0) {$\bullet$};
\node  at (-1,0) {$\dots$};
\node [below] at (0,0) {$\scriptstyle a$};
\node at (1,0) {$\times$};
\node [below] at (1,0) {$\scriptstyle a$};
\node at (2,0) {$\dots$};
\node at (4,0) {$\stackrel{r_{\beta}}{\longleftrightarrow}$};
\node at (6,0) {$\dots$};

\node  at (7,0) {$\times$};
\node [below] at (7,0) {$\scriptstyle {a+1}$};
\node at (8,0) {$\bullet$};
\node [below] at (8,0) {$\scriptstyle {a+1}$};
\node  at (9,0) {$\dots$};
\end{tikzpicture}
\end{aligned}
\end{equation}while all other nodes and entries are unchanged.

\subsection{Weight diagrams}

Weight diagrams are a convenient way to work with the highest weight
of a module with respect to the standard choice of simple roots. They
were introduced by Brundan and Stroppel in \cite{BS} and were used
to give algorithmic character formulas for basic classical Lie superalgebras
in \cite{GS} (see also \cite[5.1]{SZ2}). 

Let $\lambda\in\mathbb{P}^{+}$ and write $\lambda_{st}^{\rho}$ as
in (\ref{eq:lambda rho expansion}). On the $\mathbb{Z}$-lattice,
put $\times$ above $t$ if $t\in\left\{ a_{i}\right\} \cap\left\{ b_{j}\right\} $,
put $>$ above $t$ if $t\in\left\{ a_{i}\right\} \backslash\left\{ b_{j}\right\} $,
and put $<$ above $t$ if $t\in\left\{ b_{i}\right\} \backslash\left\{ a_{j}\right\} $.
If $t\not\in\{a_{i}\}\cup\{b_{j}\}$, then we refer to the place holder
above $t$ as an \emph{empty spot}. 
\begin{example}
If 
\begin{eqnarray*}
\lambda_{st}^{\rho} & = & 10\varepsilon_{1}+9\varepsilon_{2}+7\varepsilon_{3}+5\varepsilon_{4}+4\varepsilon_{5}-\delta_{1}-4\delta_{2}-6\delta_{3}-7\delta_{4}
\end{eqnarray*}
then the corresponding weight diagram $D_{\lambda}$ is
\begin{equation}
\stackrel{}{-1}\quad\stackrel{}{0}\quad\stackrel{<}{1}\quad\stackrel{}{2}\quad\stackrel{}{3}\quad\stackrel{\times}{4}\quad\stackrel{>}{5}\quad\stackrel{<}{6}\quad\stackrel{\times}{7}\quad\stackrel{}{8}\quad\stackrel{>}{9}\quad\stackrel{>}{10}\quad\stackrel{}{11}\quad\stackrel{}{12}.\label{eq:totally connected weight diagram}
\end{equation}

\end{example}

\subsection{Characters for Category $\mathcal{O}$}

Let $M$ be a module from the BGG category $\mathcal{O}$ \cite[8.2.3]{M}.
Then $M$ has a weight space decomposition $M=\oplus_{\mu\in\mathfrak{h}^{*}}M_{\mu}$,
where $M_{\mu}=\{x\in M\mid h.x=\mu(h)x\text{ for all }h\in\mathfrak{h}^{*}\}$,
and the \emph{character of $M$} is by definition $\mathrm{ch}\ M=\sum_{\mu\in\mathfrak{\mathfrak{h}^{*}}}dim\ M_{\mu}\ e^{\mu}.$ 

Denote by $\mathcal{\mathcal{E}}$ the algebra of rational functions
$\mathbb{Q}(e^{\nu},\ \nu\in\mathfrak{h}^{*})$. The group $W$ acts
on $\mathcal{E}$ by mapping $e^{\nu}$ to $e^{w(\nu)}$. Corresponding
to a choice of positive roots $\Delta^{+}$, the \emph{Weyl denominator
of $\mathfrak{g}$} is defined to be 
\[
R=\frac{{\Pi_{\alpha\in\Delta_{\bar{0}}^{+}}(1-e^{-\alpha})}}{{\Pi_{\alpha\in\Delta_{\bar{1}}^{+}}(1+e^{-\alpha})}}.
\]
Then $e^{\rho}R$ is $W$-skew-invariant, i.e. $w(e^{\rho}R)=(-1)^{l(w)}e^{\rho}R$,
and $\mathrm{ch}\ L(\lambda)$ is $W$-invariant for $\lambda\in P^{+}$.
The character of a Verma module is $\mathrm{ch}\ M(\lambda)=e^{\lambda}R$.
The character of the Kac module for $\lambda\in P^{+}$ is 
\begin{equation}
\mathrm{ch\ }\overline{L}(\lambda)=\frac{1}{e^{\rho}R}\sum_{w\in W}(-1)^{l(w)}w(e^{\lambda^{\rho}})\label{eq:Kac character}
\end{equation}

For $X\in\mathcal{E}$, we define 
\[
\mathcal{F}_{W}(X):=\sum_{w\in W}(-1)^{l(w)}w(X).
\]

\begin{lem}
\label{lem:reflection yields F=00003D0}Let $X\in\mathcal{E}$. If
$\sigma(X)=X$ for some reflection $\sigma\in W$, then $F_{W}(X)=0$.\end{lem}
\begin{proof}
$2\mathcal{F}_{W}(X)=\mathcal{F}_{W}(X)+\mathcal{F}_{W}(\sigma(X))=\mathcal{F}_{W}(X)+(-1)^{l(\sigma)}\mathcal{F}_{W}(X)=0$.
\end{proof}

\section{Highest weights of  KW-modules}

We describe the highest weights and arc diagrams of  KW-modules with
respect to different choices of simple roots.

\subsection{An admissible choice of simple roots. }

Let us characterize the highest weight $\lambda$ of a  KW-module
$L$ with respect to an admissible choice of simple roots $\pi$ and
a $\lambda^{\rho}$-maximal isotropic subset $S\subset\pi$.
\begin{lem}
\label{lem:Tame structure}Consider two arcs in the arc diagram of
$\lambda^{\rho}$ with no arcs between them. Up to a reflection along
these arcs, the corresponding subdiagram has one of the following
forms.\\
\begin{eqnarray}
&\label{subdiagram 1 of admissible}\begin{tikzpicture} 
%diagram
\node  at (1,0) {$\bullet$};
\node [below] at (1,0) {$\scriptstyle{a+k}$};
\node  at (2,0) {$\times$};
\node [below] at (2,0) {$\scriptstyle{a+k}$};
\node  at (3,0) {$\bullet$};
\node [below] at (3,0) {$\scriptstyle{a+k}$};
\node  at (4,0) {$\bullet$};
\node [below] at (4,0) {$\scriptstyle{a+k-1}$};
\node [below] at (5,0) {$\dots$};
\node  at (6,0) {$\bullet$};
\node [below] at (6,0) {$\scriptstyle{a+1}$};
\node  at (7,0) {$\bullet$};
\node [below] at (7,0) {$\scriptstyle{a}$};
\node  at (8,0) {$\times$};
\node [below] at (8,0) {$\scriptstyle{a}$};
%arcs
\draw (2,0.2) arc [radius=1, start angle=60, end angle= 120];
\draw (8,0.2) arc [radius=1, start angle=60, end angle= 120];
\end{tikzpicture}\\
&\label{subdiagram 2 of admissible}\begin{tikzpicture} 
%diagram 
\node  at (1,0) {$\bullet$};
\node [below] at (1,0) {$\scriptstyle{a}$};
\node  at (2,0) {$\times$};
\node [below] at (2,0) {$\scriptstyle{a}$};
\node  at (3,0) {$\times$};
\node [below] at (3,0) {$\scriptstyle{a+1}$};
\node  at (4,0) {$\times$};
\node [below] at (4,0) {$\scriptstyle{a+2}$};
\node [below] at (5,0) {$\dots$};
\node  at (6,0) {$\times$};
\node [below] at (6,0) {$\scriptstyle{a+k}$};
\node  at (7,0) {$\bullet$};
\node [below] at (7,0) {$\scriptstyle{a+k}$};
\node  at (8,0) {$\times$};
\node [below] at (8,0) {$\scriptstyle{a+k}$};
%arcs
\draw (2,0.2) arc [radius=1, start angle=60, end angle= 120];
\draw (8,0.2) arc [radius=1, start angle=60, end angle= 120];
\end{tikzpicture}
\end{eqnarray}Equivalently, the neighboring $\bullet$ and $\times$ entries are
equal, and the entries between the two arcs are either all of $\bullet$-type
and are increasing consecutive integers or are all of $\times$-type
and are decreasing consecutive integers. Moreover, diagrams (\ref{subdiagram 1 of admissible})
and (\ref{subdiagram 2 of admissible}) can not both be a subdiagram
of the same admissible diagram.\end{lem}
\begin{proof}
We shall analyze the entries of $\lambda$ and $\rho$ separately.

Let us show that all the entries of $\lambda$ in the range of the
subdiagram are equal. We write $\lambda=\sum_{i=1}^{m}a_{i}\varepsilon_{i}-\sum_{i=1}^{n}b_{i}\delta_{i}$
and suppose that the arcs in diagram correspond to the simple isotropic
roots $\alpha=\varepsilon_{i}-\delta_{i'}$ and $\beta=\varepsilon_{j}-\delta_{j'}$,
with $i<j$, $i'<j'$. Since $\alpha$ and $\beta$ are simple and
isotropic, they are orthogonal to $\rho$ and hence to $\lambda$.
This implies that $a_{i}=b_{i'}$ and $a_{j}=b_{j'}$. Since $\lambda$
is dominant, $a_{i}\ge a_{i+1}\ge\dots\ge a_{j}$ and $b_{i'}\le b_{i'+1}\le\dots\le b_{j'}$.
Hence, $a_{i}=\dots=a_{j}=b_{i'}=\dots=b_{j'}$. 

It follows that two entries in the subdiagram are equal if and only
if they are equal in $\rho$. Since $\rho$ is orthogonal to all simple
isotropic roots, the $\bullet$ and $\times$ entries are equal when
adjacent. Hence, at least one of them must be connected with an arc,
since otherwise we contradict the maximality property of the arc arrangement.
Therefore, the entries between the two arcs are either all of $\bullet$-type
or all of $\times$-type.

The difference between two consecutive entries of the subdiagram of
$\lambda^{\rho}$ will be as in $\rho$, since all such entries of
$\lambda$ are equal. In particular, they should decrease (resp. increase)
by $1$ whenever we have consecutive $\bullet$-entries (resp. $\times$-entries).
\end{proof}
The following definition was introduced in \cite[Prop. 1]{MV2} for
the standard choice of simple roots.
\begin{defn}
Let $\pi$ be any set of simple roots and let $\lambda\in\mathfrak{h}^{*}$.
Write $\lambda^{\rho}=\sum_{i=1}^{m}a_{i}\varepsilon_{i}-\sum_{i=1}^{n}b_{i}\delta_{i}$.
Suppose $S_{\lambda}=\left\{ \varepsilon_{m_{i}}-\delta_{n_{i}}\right\} _{i=1}^{r}$
is a $\lambda^{\rho}$-maximal isotropic set ordered so that $a_{m_{1}}\leq a_{m_{2}}\leq\ldots\leq a_{m_{r}}$.
We say that $\lambda^{\rho}$ satisfies the \emph{interval property}
if all the integers between $a_{m_{1}}$ and $a_{m_{r}}$ (equivalently,
between $b_{n_{1}}$ and $b_{n_{r}}$) are contained in the set $\left\{ a_{i},b_{j}\right\} _{i=1,\ldotp,m;j=1,\ldots,n}$. \end{defn}
\begin{cor}
\label{cor:tame has interval property}Let $L$ be a  KW-module and
let $\lambda$ be its highest weight with respect to an admissible
choice of simple roots. Then $\lambda^{\rho}$ satisfies the interval
property. \end{cor}
\begin{rem}
Note that the interval property is not a property of a module. In
particular, if $\pi_{1}$, $\pi_{2}$ are two choices of simple roots
and $\lambda_{1},\lambda_{2}\in\mathfrak{h}^{*}$ are such that $=L_{\pi_{1}}\left(\lambda_{\pi_{1}}\right)=L_{\pi_{2}}\left(\lambda_{\pi_{2}}\right)$,
we can have that $\lambda_{\pi_{1}}^{\rho}$ satisfies the interval
property but $\lambda_{\pi_{2}}^{\rho}$ doesn't. For example,

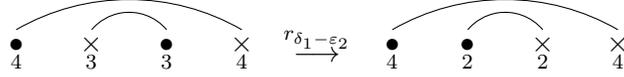
\begin{equation*}\begin{tikzpicture} 
%diagram
\node [below] at (0,0) {$\scriptstyle 4$};
\node at (0,0) {$\bullet$};
\node at (1,0) {$\times$};
\node [below] at (1,0) {$\scriptstyle  3$};
\node at (2,0) {$\bullet$};
\node [below] at (2,0) {$\scriptstyle  3$};
\node at (3,0) {$\times$};
\node [below] at (3,0) {$\scriptstyle 4$};
\node at (4,0){$\stackrel{r_{\delta_{1}-\varepsilon_{2}}}{\longrightarrow}$};
%arcs
\draw (2,0.2) arc [radius=0.65, start angle=40, end angle= 140];
\draw (3,0.2) arc [radius=3, start angle=60, end angle= 120];
\node at (5,0) {$\bullet$};
\node [below] at (5,0) {$\scriptstyle  4$};
\node at (6,0) {$\bullet$};
\node [below] at (6,0) {$\scriptstyle  2$};
\node at (7,0) {$\times$};
\node [below] at (7,0) {$\scriptstyle  2$};
\node at (8,0) {$\times$};
\node [below] at (8,0) {$\scriptstyle  4$};
%arcs
\draw (7,0.2) arc [radius=0.65, start angle=40, end angle= 140];
\draw (8,0.2) arc [radius=3, start angle=60, end angle= 120];
\end{tikzpicture}\end{equation*} 
\end{rem}

\subsection{Totally connected weights in the standard choice.}

We prove a criterion for a module to be a KW-module given its highest
weight $\lambda_{st}$ with respect to the standard choice of simple
roots $\pi_{st}$. The following definition is equivalent to the one
given in \cite{SZ1}, which was first observed in \cite{MV2}.
\begin{defn}
\label{def: totally connected}Let $\lambda_{st}\in\mathbb{P}^{+}$.
We say that $\lambda_{st}$ is \emph{totally connected} if $\lambda_{st}^{\rho}$
satisfies the interval property with respect to $\pi_{st}$. \end{defn}
\begin{rem}
In terms of weight diagrams, this is equivalent to the condition that
there are no empty spots between the $\times$'s. For example, diagram
(\ref{eq:totally connected weight diagram}) is totally connected
whereas diagram (\ref{eq:example of a non totally connnected diagram})
is not.\end{rem}
\begin{example}
\label{Example:covariant modules have t.c. highest weight}The highest
weight of a covariant module is totally connected. Indeed, such a
module corresponds to a partition $\mu$ of $k$ that lies in the
$\left(m,n\right)$-hook, i.e. $\mu_{m+1}\leq n$ \cite{BR,S}. The
corresponding covariant module is $L_{\pi_{st}}\left(\lambda\right)$,
where\\
\[
\lambda=\mu_{1}\varepsilon_{1}+\ldots+\mu_{m}\varepsilon_{m}+\tau_{1}\delta_{1}+\ldots+\tau_{n}\delta_{n}
\]
$\tau_{j}=\mathrm{max}\{0,\mu_{j}'-m\}$ for $j=1,\ldots,n$, and
$\mu_{j}'$ is the length of the $j$ -th column (see for example
\cite[Section V]{VHKT}). Then $\tau_{j}=0$ for $j>\mu_{m}$ and
the arc diagram of $\lambda^{\rho}$ is  \begin{equation*}\begin{tikzpicture} 
%diagram
\node at (0,0) {$\bullet$};
\node [below] at (0,0) {$\scriptstyle {\mu_1+m}$};
\node  at (1,0) {$\dots$};
%\node [below] at (1,0) {$\scriptstyle \mu_2$};
\node at (2,0) {$\bullet$};
\node [below] at (2,0) {$\scriptstyle \mu_m+1$};
\node at (3.5,0) {$\times$};
\node [below] at (3.5,0) {$\scriptstyle{-\tau_1+1}$};
\node at (4.5,0) {$\ldots$};

\node  at (5.5,0) {$\times$};
\node [below] at (5.5,0) {$\scriptstyle{-\tau_{\mu_m}+\mu_m}$};
\node  at (7,0) {$\times$};
\node [below] at (7,0) {$\scriptstyle {\mu_m+1}$};
\node  at (8,0) {$\times$};
\node [below] at (8,0) {$\scriptstyle {\mu_m+1}$};
\node at (9,0) {$\dots$};

\node  at (10,0) {$\times$};
\node[below] at (10,0) {$\scriptstyle n$};

\end{tikzpicture}\end{equation*} Since $\mu_{m}+1>-\tau_{\mu_{m}}+\mu_{m}$, there are no arcs connected
to the first $\mu_{m}$ $\times$-nodes. Since the rest of the $\times$-entries
are consecutive integers, $\lambda^{\rho}$ satisfies the interval
property and hence $\lambda$ is totally connected.
\end{example}

\begin{thm}
\label{thm:tame equivalent totally connected}The finite dimensional
simple module $L$ is  a KW-module if and only if its highest weight
with respect to the standard choice of simple roots is totally connected.
\begin{proof}[Proof of Theorem \ref{thm:tame equivalent totally connected} ``$\Rightarrow$'']
We start with an arc diagram that corresponds to an admissible choice
of simple roots, that is, all arcs are short, and we move to the standard
arc diagram by applying a sequence of odd reflections which preserve
the interval property. We begin with reflecting along all the arcs
which are of $\times-\bullet$ type, and get another admissible choice
of simple roots, for which the interval property holds by Corollary
\ref{cor:tame has interval property}. 

Now we push all the $\times$-entries to the right one at a time,
starting with the right most $\times$-entry. All of our reflections
are along consecutive $\times$ and $\bullet$ entries, and at each
reflection there are several cases. If the entries below the $\times$
and the $\bullet$ are not equal, then the reflection does change
them and the interval property is clearly preserved. 

If the $\times$ and $\bullet$ entries are equal then at least one
of them is connected to an arc, since otherwise the number of arcs
could be increased contradicting the maximality property of the arc
arrangement. So there are three possibilities, namely, either the
$\bullet$-entry or the $\times$-entry is connected to an arc, or
both. In each case, we reflect at the consecutive $\times-\bullet$
entries and then arrange the arcs to be of $\bullet-\times$ type
as follows.\begin{equation*}
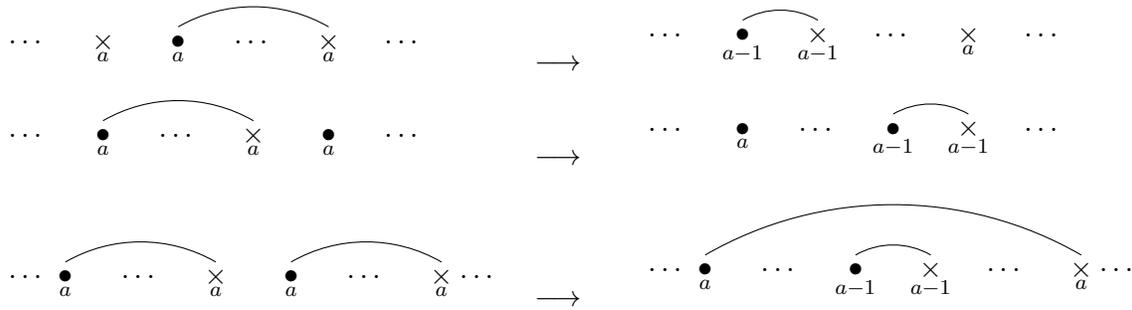
\begin{aligned}
%\begin{eqnarray*}
&\begin{tikzpicture} 
%diagram
\node at (0,0) {$\dots$};
\node at (1,0) {$\times$};
\node [below] at (1,0) {$\scriptstyle a$};
\node at (2,0) {$\bullet$};
\node [below] at (2,0) {$\scriptstyle a$};
\node at (3,0) {$\dots$};
\node at (4,0) {$\times$};
\node [below] at (4,0) {$\scriptstyle a$};
\node  at (5,0) {$\dots$};
%arcs
\draw (2,0.2) arc [radius=2, start angle=120, end angle= 60];
\end{tikzpicture} &\longrightarrow\qquad&\begin{tikzpicture} 
%diagram
\node at (0,0) {$\dots$};
\node at (1,0) {$\bullet$};
\node [below] at (1,0) {$\scriptstyle {a-1}$};
\node at (2,0) {$\times$};
\node [below] at (2,0) {$\scriptstyle {a-1}$};
\node at (3,0) {$\dots$};
\node at (4,0) {$\times$};
\node  [below] at (4,0) {$\scriptstyle a$};
\node  at (5,0) {$\dots$};
%arcs
\draw (1,0.2) arc [radius=1, start angle=120, end angle= 60];
\end{tikzpicture} \\
&\begin{tikzpicture} 
%diagram
\node at (0,0) {$\dots$};
\node at (1,0) {$\bullet$};
\node [below] at (1,0) {$\scriptstyle a$};
\node at (2,0) {$\dots$};
\node [below] at (3,0) {$\scriptstyle a$};
\node at (3,0) {$\times$};
\node at (4,0) {$\bullet$};
\node [below] at (4,0) {$\scriptstyle a$};
\node  at (5,0) {$\dots$};
%arcs
\draw (1,0.2) arc [radius=2, start angle=120, end angle= 60];
\end{tikzpicture} & \longrightarrow\qquad&\begin{tikzpicture} 
%diagram
\node at (0,0) {$\dots$};
\node at (1,0) {$\bullet$};
\node [below] at (1,0) {$\scriptstyle a$};
\node at (2,0) {$\dots$};
\node [below] at (3,0) {$\scriptstyle {a-1}$};
\node at (3,0) {$\bullet$};
\node at (4,0) {$\times$};
\node [below] at (4,0) {$\scriptstyle {a-1}$};
\node  at (5,0) {$\dots$};
%arcs
\draw (3,0.2) arc [radius=1, start angle=120, end angle= 60];
\end{tikzpicture} \\
&\begin{tikzpicture} 
%diagram
\node at (0.5,0) {$\dots$};
\node at (1,0) {$\bullet$};
\node [below] at (1,0) {$\scriptstyle a$};
\node at (2,0) {$\dots$};
\node [below] at (3,0) {$\scriptstyle a$};
\node at (3,0) {$\times$};
\node at (4,0) {$\bullet$};
\node [below] at (4,0) {$\scriptstyle a$};
\node  at (5,0) {$\dots$};
\node at (6,0) {$\times$};
\node [below] at (6,0) {$\scriptstyle a$};
\node  at (6.5,0) {$\dots$};
%arcs
\draw (1,0.2) arc [radius=2, start angle=120, end angle= 60];
\draw (4,0.2) arc [radius=2, start angle=120, end angle= 60];
\end{tikzpicture}& \longrightarrow\qquad&\begin{tikzpicture} 
%diagram
\node at (0.5,0) {$\dots$};
\node at (1,0) {$\bullet$};
\node [below] at (1,0) {$\scriptstyle a$};
\node at (2,0) {$\dots$};
\node [below] at (3,0) {$\scriptstyle {a-1}$};
\node at (3,0) {$\bullet$};
\node at (4,0) {$\times$};
\node [below] at (4,0) {$\scriptstyle {a-1}$};
\node  at (5,0) {$\dots$};
\node at (6,0) {$\times$};
\node [below] at (6,0) {$\scriptstyle a$};
\node  at (6.5,0) {$\dots$};
%arcs
\draw (1,0.2) arc [radius=5, start angle=120, end angle= 60];
\draw (3,0.2) arc [radius=1, start angle=120, end angle= 60];
\end{tikzpicture}
%\end{eqnarray*}
\end{aligned}
\end{equation*}

In each case, an atypical entry of $a-1$ was added, but the interval
property is preserved since we still have an $a$ in the diagram.
The maximality property is also preserved since the new diagram has
the same number of arcs as the old diagram. Moreover, the arcs of
the new diagram are again non-intersecting and of $\bullet-\times$
type. We proceed to the next $\times$, and continue until each $\times$
is to the right of every $\bullet$.
\end{proof}
\end{thm}

To prove the other direction, we introduce an algorithm to move from
a totally connected weight in the standard choice of simple roots
to an admissible choice of simple roots, which we call the \emph{shortening
algorithm}. Let's first illustrate our algorithm with an example.
\begin{example}
\label{exp:taming}Let us take the following totally connected weight
and show a sequence of odd reflections that bring it to an admissible
choice of simple roots.
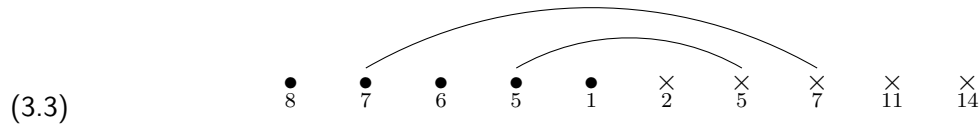
\begin{equation}\label{diagram example in standard choice}\begin{tikzpicture} 
%diagram
\node at (0,0) {$\bullet$};
\node [below] at (0,0) {$\scriptstyle 8$};
\node at (1,0) {$\bullet$};
\node [below] at (1,0) {$\scriptstyle 7$};
\node at (2,0) {$\bullet$};
\node [below] at (2,0) {$\scriptstyle 6$};
\node at (3,0) {$\bullet$};
\node [below] at (3,0) {$\scriptstyle 5$};
\node at (4,0) {$\bullet$};
\node [below] at (4,0) {$\scriptstyle 1$};
\node  at (5,0) {$\times$};
\node [below] at (5,0) {$\scriptstyle 2$};
\node  at (6,0) {$\times$};
\node [below] at (6,0) {$\scriptstyle 5$};
\node  at (7,0) {$\times$};
\node[below] at (7,0) {$\scriptstyle 7$};
\node  at (8,0) {$\times$};
\node[below] at (8,0) {$\scriptstyle 11$};
\node  at (9,0) {$\times$};
\node[below] at (9,0) {$\scriptstyle 14$};

%arcs
\draw (3,0.2) arc [radius=3, start angle=120, end angle= 60];
\draw (1,0.2) arc [radius=6, start angle=120, end angle= 60];

\end{tikzpicture}\end{equation}Since the innermost entries $1$ and $2$ are different from all the
other entries, using the odd reflection defined in (\ref{odd reflection of diagrams})
we can move them outside the arcs to the right and left, respectively.
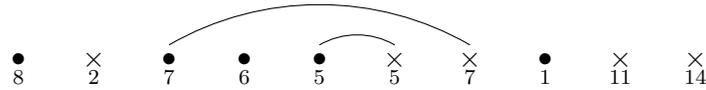
\begin{equation*}\begin{tikzpicture} 
%diagram
\node at (0,0) {$\bullet$};
\node [below] at (0,0) {$\scriptstyle 8$};
\node  at (1,0) {$\times$};
\node [below] at (1,0) {$\scriptstyle 2$};
\node at (2,0) {$\bullet$};
\node [below] at (2,0) {$\scriptstyle 7$};
\node at (3,0) {$\bullet$};
\node [below] at (3,0) {$\scriptstyle 6$};
\node at (4,0) {$\bullet$};
\node [below] at (4,0) {$\scriptstyle 5$};
\node  at (5,0) {$\times$};
\node [below] at (5,0) {$\scriptstyle 5$};
\node  at (6,0) {$\times$};
\node [below] at (6,0) {$\scriptstyle 7$};
\node at (7,0) {$\bullet$};
\node[below] at (7,0) {$\scriptstyle 1$};
\node  at (8,0) {$\times$};
\node[below] at (8,0) {$\scriptstyle 11$};
\node  at (9,0) {$\times$};
\node[below] at (9,0) {$\scriptstyle 14$};

%arcs
\draw (2,0.2) arc [radius=4, start angle=120, end angle= 60];
\draw (4,0.2) arc [radius=1, start angle=120, end angle= 60];

\end{tikzpicture}\end{equation*}Next, we apply the odd reflection $s_{\varepsilon_{4}-\delta_{2}}$,
and then choose the arc arrangement to be of $\bullet-\times$ type
so that we can move the extra $6$ outside of the arcs.
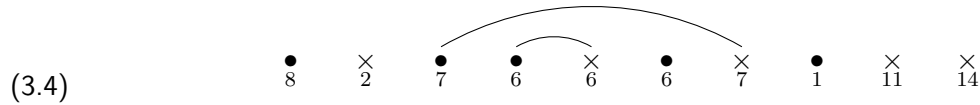
\begin{equation}\label{diagram example after reflecting}\begin{tikzpicture} 
%diagram
\node at (0,0) {$\bullet$};
\node [below] at (0,0) {$\scriptstyle 8$};
\node  at (1,0) {$\times$};
\node [below] at (1,0) {$\scriptstyle 2$};
\node at (2,0) {$\bullet$};
\node [below] at (2,0) {$\scriptstyle 7$};
\node at (3,0) {$\bullet$};
\node [below] at (3,0) {$\scriptstyle 6$};
\node  at (4,0) {$\times$};
\node [below] at (4,0) {$\scriptstyle 6$};
\node at (5,0) {$\bullet$};
\node [below] at (5,0) {$\scriptstyle 6$};
\node  at (6,0) {$\times$};
\node [below] at (6,0) {$\scriptstyle 7$};
\node at (7,0) {$\bullet$};
\node[below] at (7,0) {$\scriptstyle 1$};
\node  at (8,0) {$\times$};
\node[below] at (8,0) {$\scriptstyle 11$};
\node  at (9,0) {$\times$};
\node[below] at (9,0) {$\scriptstyle 14$};

%arcs
\draw (2,0.2) arc [radius=4, start angle=120, end angle= 60];
\draw (3,0.2) arc [radius=1, start angle=120, end angle= 60];

\end{tikzpicture}\end{equation} Then we move the $6$ $\bullet$-entry to the right outside of the
arcs.
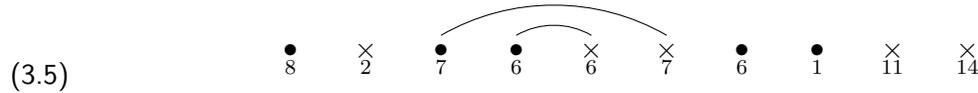
\begin{equation}\label{diagram example after one step}\begin{tikzpicture} 
%diagram
\node at (0,0) {$\bullet$};
\node [below] at (0,0) {$\scriptstyle 8$};
\node  at (1,0) {$\times$};
\node [below] at (1,0) {$\scriptstyle 2$};
\node at (2,0) {$\bullet$};
\node [below] at (2,0) {$\scriptstyle 7$};
\node at (3,0) {$\bullet$};
\node [below] at (3,0) {$\scriptstyle 6$};
\node  at (4,0) {$\times$};
\node [below] at (4,0) {$\scriptstyle 6$};
\node  at (5,0) {$\times$};
\node [below] at (5,0) {$\scriptstyle 7$};
\node at (6,0) {$\bullet$};
\node [below] at (6,0) {$\scriptstyle 6$};
\node at (7,0) {$\bullet$};
\node[below] at (7,0) {$\scriptstyle 1$};
\node  at (8,0) {$\times$};
\node[below] at (8,0) {$\scriptstyle 11$};
\node  at (9,0) {$\times$};
\node[below] at (9,0) {$\scriptstyle 14$};

%arcs
\draw (2,0.2) arc [radius=3, start angle=120, end angle= 60];
\draw (3,0.2) arc [radius=1, start angle=120, end angle= 60];

\end{tikzpicture}\end{equation}Finally, we apply $s_{\varepsilon_{3}-\delta_{2}}$, and then arrange
the arcs to be short, obtaining an admissible choice of simple roots.
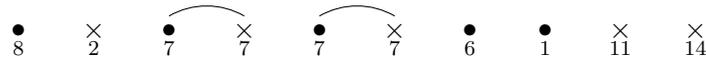
\begin{equation*}\begin{tikzpicture} 
%diagram
\node at (0,0) {$\bullet$};
\node [below] at (0,0) {$\scriptstyle 8$};
\node  at (1,0) {$\times$};
\node [below] at (1,0) {$\scriptstyle 2$};
\node at (2,0) {$\bullet$};
\node [below] at (2,0) {$\scriptstyle 7$};
\node  at (3,0) {$\times$};
\node [below] at (3,0) {$\scriptstyle 7$};
\node at (4,0) {$\bullet$};
\node [below] at (4,0) {$\scriptstyle 7$};
\node  at (5,0) {$\times$};
\node [below] at (5,0) {$\scriptstyle 7$};
\node at (6,0) {$\bullet$};
\node [below] at (6,0) {$\scriptstyle 6$};
\node at (7,0) {$\bullet$};
\node[below] at (7,0) {$\scriptstyle 1$};
\node  at (8,0) {$\times$};
\node[below] at (8,0) {$\scriptstyle 11$};
\node  at (9,0) {$\times$};
\node[below] at (9,0) {$\scriptstyle 14$};

%arcs
\draw (2,0.2) arc [radius=1, start angle=120, end angle= 60];
\draw (4,0.2) arc [radius=1, start angle=120, end angle= 60];

\end{tikzpicture}\end{equation*}\end{example}
\begin{proof}[Proof of Theorem \ref{thm:tame equivalent totally connected} ``$\Leftarrow$'']
We give an algorithm to move from a totally connected weight $\lambda_{st}$
of a finite dimensional simple module to an admissible choice of simple
roots using a sequence of odd reflections. The main idea is to push
all the typical entries which are below the arcs to the side, making
all the atypical entries the same. This will allow us to choose an
arc arrangement with only short arcs.

We have from Remark \ref{rem: arc diagram entries} that adjacent
$\bullet$-entries of $\lambda_{\pi}^{\rho}$ are strictly decreasing,
while adjacent $\times$-entries of $\lambda_{\pi}^{\rho}$ are strictly
increasing. Hence, in the standard arc diagram all equalities between
entries correspond to arcs, and the arcs are nested in each other.
Moreover, all of the entries below the innermost arc are typical and
are different from the rest of the entries. Before applying the algorithm,
we move these entries outside of the arcs by pushing the $\times$'s
to the left and the $\bullet$'s to the right, which makes the innermost
arc short and of $\bullet-\times$ type. 

We begin the algorithm by reflecting along the innermost arc and then
we arrange the arcs to be of $\bullet-\times$ type. Due to the interval
property there are three possibilities: either there is an $a+1$
$\bullet$-entry on the left, an $a+1$ $\times$-entry on the right,
or both and in which case they must be connected by an arc. 

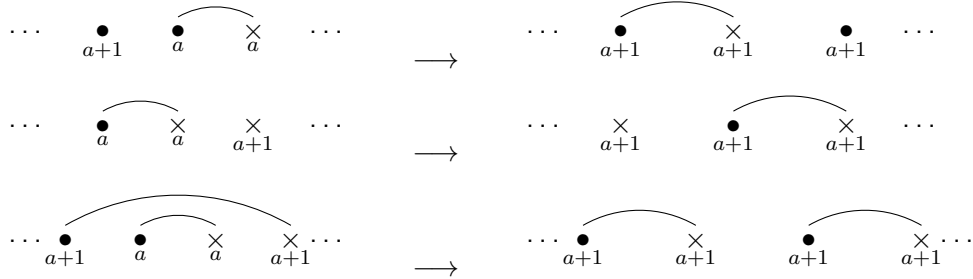
\begin{equation*}
\begin{aligned}
&\begin{tikzpicture} 
%diagram
\node  at (-1,0) {\dots};
\node at (0,0) {$\bullet$};
\node [below] at (0,0) {$\scriptstyle{a+1}$};
\node at (1,0) {$\bullet$};
\node [below] at (1,0) {$\scriptstyle a$};
\node at (2,0) {$\times$};
\node [below] at (2,0) {$\scriptstyle a$};
\node  at (3,0) {\dots};
%arcs
\draw (1,0.2) arc [radius=1, start angle=120, end angle= 60];
\end{tikzpicture}&\quad\longrightarrow\qquad &\begin{tikzpicture} 
%diagram
\node at (-1,0) {\dots};
\node at (0,0) {$\bullet$};
\node [below] at (0,0) {$\scriptstyle{a+1}$};
\node at (1.5,0) {$\times$};
\node [below] at (1.5,0) {$\scriptstyle{a+1}$};
\node at (3,0) {$\bullet$};
\node [below] at (3,0) {$\scriptstyle{a+1}$};
\node at (4,0) {\dots};
%arcs
\draw (0,0.2) arc [radius=1.5, start angle=120, end angle= 60];
\end{tikzpicture}\\
&\begin{tikzpicture} 
%diagram
\node  at (-1,0) {\dots};
\node at (0,0) {$\bullet$};
\node [below] at (0,0) {$\scriptstyle a$};
\node at (,0) {$\times$};1
\node [below] at (1,0) {$\scriptstyle a$};
\node at (2,0) {$\times$};
\node [below] at (2,0) {$\scriptstyle{a+1}$};
\node  at (3,0) {\dots};
%arcs
\draw (0,0.2) arc [radius=1, start angle=120, end angle= 60];
\end{tikzpicture} 
&\quad\longrightarrow\qquad 
&\begin{tikzpicture} 
%diagram
\node at (-1,0) {\dots};
\node at (0,0) {$\times$};
\node [below] at (0,0) {$\scriptstyle{a+1}$};
\node at (1.5,0) {$\bullet$};
\node [below] at (1.5,0) {$\scriptstyle{a+1}$};
\node at (3,0) {$\times$};
\node [below] at (3,0) {$\scriptstyle{a+1}$};
\node at (4,0) {\dots};
%arcs
\draw (1.5,0.2) arc [radius=1.5, start angle=120, end angle= 60];
\end{tikzpicture}\\
&\begin{tikzpicture} 
%diagram
\node  at (-0.5,0) {\dots};
\node at (0,0) {$\bullet$};
\node [below] at (0,0) {$\scriptstyle{a+1}$};
\node at (1,0) {$\bullet$};
\node [below] at (1,0) {$\scriptstyle a$};
\node at (2,0) {$\times$};
\node [below] at (2,0) {$\scriptstyle a$};
\node at (3,0) {$\times$};
\node [below] at (3,0) {$\scriptstyle{a+1}$};
\node  at (3.5,0) {\dots};
%arcs
\draw (0,0.2) arc [radius=3, start angle=120, end angle= 60];
\draw (1,0.2) arc [radius=1, start angle=120, end angle= 60];
\end{tikzpicture}&\quad\longrightarrow\qquad &\begin{tikzpicture} 
%diagram
\node  at (-.5,0) {\dots};
\node at (0,0) {$\bullet$};
\node [below] at (0,0) {$\scriptstyle{a+1}$};
\node at (1.5,0) {$\times$};
\node [below] at (1.5,0) {$\scriptstyle{a+1}$};
\node at (3,0) {$\bullet$};
\node [below] at (3,0) {$\scriptstyle{a+1}$};
\node at (4.5,0) {$\times$};
\node [below] at (4.5,0) {$\scriptstyle{a+1}$};
\node  at (5,0) {\dots};
%arcs
\draw (0,0.2) arc [radius=1.5, start angle=120, end angle= 60];
\draw (3,0.2) arc [radius=1.5, start angle=120, end angle= 60];
\end{tikzpicture}
\end{aligned}
\end{equation*} \\
After the reflection, in the first case we push the unmatched $a+1$
$\bullet$-entry to the right outside of the arcs, and in the second
case we push the unmatched $a+1$ $\times$-entry to the left outside
of the arcs.

This will be the base of our induction. Suppose that after $k$ steps,
all the atypical entries below $a+k+1$ are now equal to $a+k$ and
are paired by short arcs, and all other entries which are below an
arc remained as in the original diagram of $\lambda_{st}^{\rho}$.
Due to the interval property there are now three possibilities, namely,
either there is an $a+k+1$ $\bullet$-entry on the left, or an $a+k+1$
$\times$-entry on the right, or both and in which case they must
be connected by an arc. 

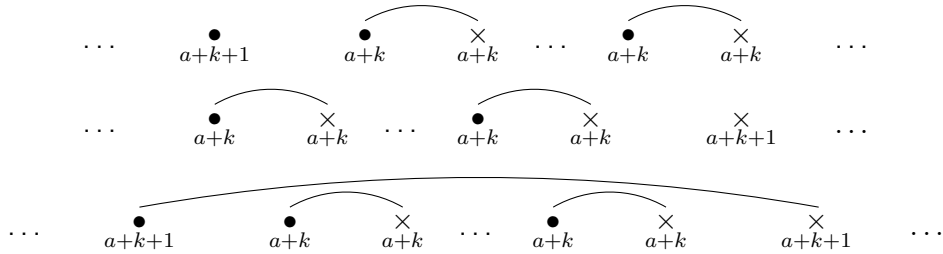
\begin{eqnarray*}
&\begin{tikzpicture} 
%diagram
\node [below] at (-1.5,0) {\dots};
\node at (0,0) {$\bullet$};
\node [below] at (0,0) {$\scriptstyle{a+k+1}$};
\node at (2,0) {$\bullet$};
\node [below] at (2,0) {$\scriptstyle{a+k}$};
\node at (3.5,0) {$\times$};
\node [below] at (3.5,0) {$\scriptstyle{a+k}$};
\node [below] at (4.5,0) {\dots};
\node at (5.5,0) {$\bullet$};
\node [below] at (5.5,0) {$\scriptstyle{a+k}$};
\node at (7,0) {$\times$};
\node [below] at (7,0) {$\scriptstyle{a+k}$};
\node [below] at (8.5,0) {\dots};
%arcs
\draw (2,0.2) arc [radius=1.5, start angle=120, end angle= 60];
\draw (5.5,0.2) arc [radius=1.5, start angle=120, end angle= 60];
\end{tikzpicture}\\
&\begin{tikzpicture} 
%diagram
\node [below] at (0,0) {\dots};
\node at (1.5,0) {$\bullet$};
\node [below] at (1.5,0) {$\scriptstyle{a+k}$};
\node at (3,0) {$\times$};
\node [below] at (3,0) {$\scriptstyle{a+k}$};
\node [below] at (4,0) {\dots};
\node at (5,0) {$\bullet$};
\node [below] at (5,0) {$\scriptstyle{a+k}$};
\node at (6.5,0) {$\times$};
\node [below] at (6.5,0) {$\scriptstyle{a+k}$};
\node at (8.5,0) {$\times$};
\node [below] at (8.5,0) {$\scriptstyle{a+k+1}$};
\node[below] at (10,0){$\dots$};
%arcs
\draw (1.5,0.2) arc [radius=1.5, start angle=120, end angle= 60];
\draw (5,0.2) arc [radius=1.5, start angle=120, end angle= 60];
\end{tikzpicture}\\
&\begin{tikzpicture} 
%diagram
\node [below] at (-2,0) {\dots};
\node at (-.5,0) {$\bullet$};
\node [below] at (-.5,0) {$\scriptstyle{a+k+1}$};
\node at (1.5,0) {$\bullet$};
\node [below] at (1.5,0) {$\scriptstyle{a+k}$};
\node at (3,0) {$\times$};
\node [below] at (3,0) {$\scriptstyle{a+k}$};
\node [below] at (4,0) {\dots};
\node at (5,0) {$\bullet$};
\node [below] at (5,0) {$\scriptstyle{a+k}$};
\node at (6.5,0) {$\times$};
\node [below] at (6.5,0) {$\scriptstyle{a+k}$};
\node at (8.5,0) {$\times$};
\node [below] at (8.5,0) {$\scriptstyle{a+k+1}$};
\node[below] at (10,0){$\dots$};
%arcs
\draw (1.5,0.2) arc [radius=1.5, start angle=120, end angle= 60];
\draw (5,0.2) arc [radius=1.5, start angle=120, end angle= 60];
\draw (-0.5,0.2) arc [radius=26, start angle=100, end angle= 80];
\end{tikzpicture}
\end{eqnarray*}  \\
We reflect along the short arcs and arrange the arcs to get the following
three diagrams, respectively.\\
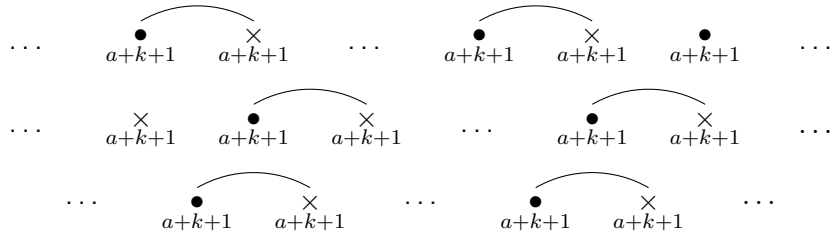
\begin{eqnarray*}
&\begin{tikzpicture} 
%diagram
\node [below] at (-1.5,0) {\dots};
\node at (0,0) {$\bullet$};
\node [below] at (0,0) {$\scriptstyle{a+k+1}$};
\node at (1.5,0) {$\times$};
\node [below] at (1.5,0) {$\scriptstyle{a+k+1}$};
\node [below] at (3,0) {\dots};
\node at (4.5,0) {$\bullet$};
\node [below] at (4.5,0) {$\scriptstyle{a+k+1}$};
\node at (6,0) {$\times$};
\node [below] at (6,0) {$\scriptstyle{a+k+1}$};
\node at (7.5,0) {$\bullet$};
\node [below] at (7.5,0) {$\scriptstyle{a+k+1}$};
\node [below] at (9,0) {\dots};
%arcs
\draw (0,0.2) arc [radius=1.5, start angle=120, end angle= 60];
\draw (4.5,0.2) arc [radius=1.5, start angle=120, end angle= 60];
\end{tikzpicture}\\
&\begin{tikzpicture} 
%diagram
\node [below] at (0,0) {\dots};
\node at (1.5,0) {$\times$};
\node [below] at (1.5,0) {$\scriptstyle{a+k+1}$};
\node at (3,0) {$\bullet$};
\node [below] at (3,0) {$\scriptstyle{a+k+1}$};
\node at (4.5,0) {$\times$};
\node [below] at (4.5,0) {$\scriptstyle{a+k+1}$};
\node [below] at (6,0) {\dots};
\node at (7.5,0) {$\bullet$};
\node [below] at (7.5,0) {$\scriptstyle{a+k+1}$};
\node at (9,0) {$\times$};
\node [below] at (9,0) {$\scriptstyle{a+k+1}$};
\node[below] at (10.5,0){$\dots$};
%arcs
\draw (3,0.2) arc [radius=1.5, start angle=120, end angle= 60];
\draw (7.5,0.2) arc [radius=1.5, start angle=120, end angle= 60];

\end{tikzpicture}\\
&\begin{tikzpicture} 
%diagram
\node  at (-2,0) {\dots};
\node at (-0.5,0) {$\bullet$};
\node [below] at (-.5,0) {$\scriptstyle{a+k+1}$};
\node at (1,0) {$\times$};
\node [below] at (1,0) {$\scriptstyle{a+k+1}$};
\node  at (2.5,0) {\dots};
\node at (4,0) {$\bullet$};
\node [below] at (4,0) {$\scriptstyle{a+k+1}$};
\node at (5.5,0) {$\times$};
\node [below] at (5.5,0) {$\scriptstyle{a+k+1}$};
\node at (7,0){$\dots$};
%arcs
\draw (4,0.2) arc [radius=1.5, start angle=120, end angle= 60];
\draw (-0.5,0.2) arc [radius=1.5, start angle=120, end angle= 60];
\end{tikzpicture}
\end{eqnarray*}In the first case we push the unmatched $a+k+1$ $\bullet$-entry
to the right outside of the arcs, and in the second case we push the
unmatched $a+k+1$ $\times$-entry to the left outside of the arcs. 

We continue this procedure until we reach the outermost arc. After
doing the last step, all of the arcs become short and we get an admissible
choice of simple roots. Moreover, all the atypical entries are now
adjacent and equal to the largest atypical entry of $\lambda_{st}^{\rho}$.
The typical $\bullet$-entries (resp. $\times$-entries) which were
under an arc of $\lambda_{st}^{\rho}$ were pushed to the right (resp.
left). \end{proof}
\begin{rem}
\label{Rem: description of special tame} For each $\lambda_{st}^{\rho}$
which corresponds to a totally connected $\lambda$ in $\pi_{st}$,
one can immediately determine the arc diagram given by shortening
algorithm. For example, if $\lambda_{st}^{\rho}$ corresponds to the
arc diagram\\
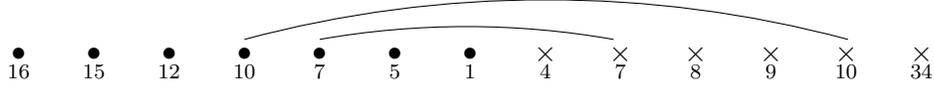
\begin{equation*}\begin{tikzpicture} 
%diagram
\node [below] at (0,0) {$\scriptstyle 16$};
\node at (0,0) {$\bullet$};
\node at (1,0) {$\bullet$};
\node [below] at (1,0) {$\scriptstyle 15$};
\node  at (2,0) {$\bullet$};
\node [below] at (2,0) {$\scriptstyle 12$};
\node  at (3,0) {$\bullet$};
\node [below] at (3,0) {$\scriptstyle 10$};
\node at (4,0) {$\bullet$};
\node [below] at (4,0) {$\scriptstyle 7$};
\node  at (5,0) {$\bullet$};
\node [below] at (5,0) {$\scriptstyle 5$};
\node  at (6,0) {$\bullet$};
\node [below] at (6,0) {$\scriptstyle 1$};
\node  at (7,0) {$\times$};
\node [below] at (7,0) {$\scriptstyle 4$};
\node  at (8,0) {$\times$};
\node [below] at (8,0) {$\scriptstyle 7$};
\node  at (9,0) {$\times$};
\node [below] at (9,0) {$\scriptstyle 8$};
\node  at (10,0) {$\times$};
\node [below] at (10,0) {$\scriptstyle 9$};
\node  at (11,0) {$\times$};
\node [below] at (11,0) {$\scriptstyle 10$};
\node  at (12,0) {$\times$};
\node [below] at (12,0) {$\scriptstyle 34$};

%arcs
\draw (3,0.2) arc [radius=15.5, start angle=105, end angle= 75];
\draw (4,0.2) arc [radius=11.25, start angle=100, end angle= 80];

\end{tikzpicture}\end{equation*}then the shortening algorithm gives 
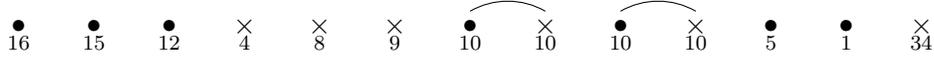
\begin{equation*}\begin{tikzpicture} 
%diagram
\node  at (0,0) {$\bullet$};
\node [below] at (0,0) {$\scriptstyle 16$};
\node  at (1,0) {$\bullet$};
\node [below] at (1,0) {$\scriptstyle 15$};
\node  at (2,0) {$\bullet$};
\node [below] at (2,0) {$\scriptstyle 12$};
\node at (3,0) {$\times$};
\node [below] at (3,0) {$\scriptstyle 4$};
\node at (4,0) {$\times$};
\node [below] at (4,0) {$\scriptstyle 8$};
\node at (5,0) {$\times$};
\node [below] at (5,0) {$\scriptstyle 9$};
\node  at (6,0) {$\bullet$};
\node [below] at (6,0) {$\scriptstyle 10$};
\node at (7,0) {$\times$};
\node [below] at (7,0) {$\scriptstyle 10$};
\node  at (8,0) {$\bullet$};
\node [below] at (8,0) {$\scriptstyle 10$};
\node at (9,0) {$\times$};
\node [below] at (9,0) {$\scriptstyle 10$};
\node  at (10,0) {$\bullet$};
\node [below] at (10,0) {$\scriptstyle 5$};
\node  at (11,0) {$\bullet$};
\node [below] at (11,0) {$\scriptstyle 1$};
\node at (12,0) {$\times$};
\node [below] at (12,0) {$\scriptstyle 34$};

%arcs
\draw (6,0.2) arc [radius=1, start angle=120, end angle= 60];
\draw (8,0.2) arc [radius=1, start angle=120, end angle= 60];

\end{tikzpicture}\end{equation*} where the typical $\bullet$-entries under an arc of $\lambda_{st}^{\rho}$
are pushed to the right outside the arcs, the typical $\times$-entries
under an arc of $\lambda_{st}^{\rho}$ are pushed to left outside
the arcs, the atypical entries are set equal to the maximal atypical
entry of $\lambda_{st}^{\rho}$ and then all arcs are chosen to be
short. \end{rem}
\begin{defn}
\label{def: special tame}We call an arc diagram for $L$ the \emph{special
arc diagram} if \begin{enumerate}
\item all arcs are short, of $\bullet-\times$  type and are adjacent;
\item all atypical entries are equal;
\item the typical nodes at each end of the diagram are organized so that the $\bullet$'s precede the $\times$'s;
\item the $\bullet$-entries are strictly decreasing left to right, except for atypical entries which are all equal;
\item the $\times$-entries are strictly increasing left to right, except for atypical entries which are all equal.
\end{enumerate}\end{defn}
\begin{rem}
The arc diagram obtained in the last step of the shortening algorithm
is a special arc diagram for $L$, since it satisfies (1)-(5). Hence,
every KW-module has a special set of simple roots, since the highest
weight of a  KW-module with respect to $\pi_{st}$ is totally connected.
Moreover, it is unique since we can apply the reverse of the shortening
algorithm to a special arc diagram to obtain the standard arc diagram
for a totally connected weight $\lambda$ of a finite dimensional
module.
\end{rem}

\section{Su-Zhang character formula for the totally connected case}

We use Brundan's algorithm to characterize KW-modules in terms of
Kazhdan-Lusztig polynomials and to prove the Su-Zhang character formula
for finite dimensional simple modules with a totally connected highest
weight in the standard choice of simple roots $\pi_{st}$. Recall
the notation $\left(\lambda_{st}^{\rho}\right)^{\Uparrow}$ and \textrm{$\left|\nu\right|_{S_{\lambda}}$}
from Section \ref{sub:Atypical-modules.}.
\begin{thm}[{Su, Zhang, \cite[4.13]{SZ1}}]
\label{thm:SZ arrow}Let $\lambda_{st}$ be a totally connected weight
with a $\lambda^{\rho}$-maximal isotropic set $S_{\lambda}$ such
that $\left|S_{\lambda}\right|=r$. Then 
\begin{equation}
e^{\rho}R\cdot\mbox{ch }L_{\pi_{st}}\left(\lambda_{st}\right)=\frac{\left(-1\right)^{\left|\left(\lambda_{st}^{\rho}\right)^{\Uparrow}-\lambda_{st}^{\rho}\right|_{S_{\lambda}}}}{r!}\mathcal{F}_{W}\left(\frac{e^{\left(\lambda_{st}^{\rho}\right)^{\Uparrow}}}{\prod_{\beta\in S_{\lambda}}\left(1+e^{-\beta}\right)}\right).\label{eq: SZ arrow formula}
\end{equation}

\end{thm}
To prove this theorem, we extend the ring $\mathcal{E}$ by adding
expansions of the elements $\frac{1}{1+e^{-\beta}}$ with \textrm{$\beta\in\Delta_{\bar{1}}^{+}$
}with respect to $\pi_{st}$ as geometric series in the domain $\left|e^{-\beta}\right|<1$.
Since $\Delta_{\bar{1}}^{+}$ is fixed by $W$, expanding commutes
with the action of $W$.

\subsection{Brundan's algorithm}

In \cite{S1}, Serganova introduced the generalized Kazhdan-Lusztig
polynomials to give a character formula for finite dimensional irreducible
representations of $\mathfrak{gl}\left(m|n\right)$. For each $\lambda$
and $\mu$ dominant integral, the Kazhdan-Lusztig polynomial $K_{\lambda,\mu}\left(q\right)$
was shown to yield the multiplicity of Kac module $\overline{L}\left(\mu\right)$
inside the simple module $L_{st}\left(\lambda\right)$ in the following
sense:
\begin{equation}
\mbox{ch }L_{\pi_{st}}\left(\lambda\right)=\sum_{\mu\in\mathfrak{h}^{*}}K_{\lambda,\mu}\left(-1\right)\mbox{ch }\overline{L}\left(\mu\right).\label{eq:composition factors of Kac modules}
\end{equation}
In this section we recall the algorithm of Brundan \cite{B} to compute
$K_{\lambda,\mu}\left(q\right)$ in terms of weight diagrams. 

We define a\emph{ right move} map from the set of (labeled) weight
diagrams to itself in two steps. Let $D_{\mu}$ be a weight diagram
for $\mu\in\mathbb{P}^{+}$, and choose a labeling of the $\times$'s
with indexing set $\{1,\ldots,r\}$. Then for each $\times$, starting
with the rightmost $\times$, ``mark'' the next empty spot to the
right of it (which is unmarked). The right move $R_{i}$ is then defined
by moving $\times_{i}$ to the empty spot it marked.
\begin{example}
\label{exp: lambda diagram}Let $D_{\mu}$ be 
\begin{equation}
\ldots\stackrel{}{-1}\quad\stackrel{}{0}\quad\stackrel{}{1}\quad\stackrel{<}{2}\quad\stackrel{}{3}\quad\stackrel{\times}{4}\quad\stackrel{>}{5}\quad\stackrel{\times}{6}\quad\stackrel{}{7}\quad\stackrel{\times}{8}\quad\stackrel{>}{9}\quad\stackrel{>}{10}\quad\stackrel{}{11}\quad\stackrel{}{12}\ldots.\label{eq:example of a non totally connnected diagram}
\end{equation}
The rightmost $\times$ is at $8$ and we mark $11$ for it. The next
$\times$ is at $6$ and so we mark $7.$ Finally for the leftmost
$\times$ we mark $12$. Then
\begin{eqnarray*}
R_{1}\left(D_{\mu}\right) & = & \ldots\stackrel{}{-1}\quad\stackrel{}{0}\quad\stackrel{<}{1}\quad\stackrel{}{2}\quad\stackrel{}{3}\quad\stackrel{}{4}\quad\stackrel{>}{5}\quad\stackrel{\times}{6}\quad\stackrel{}{7}\quad\stackrel{\times}{8}\quad\stackrel{>}{9}\quad\stackrel{>}{10}\quad\stackrel{}{11}\quad\stackrel{\times}{12}\ldots\\
R_{2}\left(D_{\mu}\right) & = & \ldots\stackrel{}{-1}\quad\stackrel{}{0}\quad\stackrel{}{1}\quad\stackrel{<}{2}\quad\stackrel{}{3}\quad\stackrel{\times}{4}\quad\stackrel{>}{5}\quad\stackrel{}{6}\quad\stackrel{\times}{7}\quad\stackrel{\times}{8}\quad\stackrel{>}{9}\quad\stackrel{>}{10}\quad\stackrel{}{11}\quad\stackrel{}{12}\ldots\\
R_{3}\left(D_{\mu}\right) & = & \ldots\stackrel{}{-1}\quad\stackrel{}{0}\quad\stackrel{<}{1}\quad\stackrel{}{2}\quad\stackrel{}{3}\quad\stackrel{\times}{4}\quad\stackrel{>}{5}\quad\stackrel{\times}{6}\quad\stackrel{}{7}\quad\stackrel{}{8}\quad\stackrel{>}{9}\quad\stackrel{>}{10}\quad\stackrel{\times}{11}\quad\stackrel{}{12}\ldots..
\end{eqnarray*}

\end{example}
Note that the weight $\mu_{i}^{\rho}$ corresponding to $R_{i}\left(D_{\mu}\right)$
does not only differ from $\mu^{\rho}$ by atypical roots. It also
has a different atypical set. In the previous example
\begin{eqnarray*}
S_{\mu} & = & \left\{ \varepsilon_{3}-\delta_{2},\varepsilon_{4}-\delta_{3},\varepsilon_{6}-\delta_{4}\right\} \\
S_{\mu_{1}} & = & \left\{ \varepsilon_{1}-\delta_{2},\varepsilon_{4}-\delta_{3},\varepsilon_{5}-\delta_{4}\right\} .
\end{eqnarray*}
 
\begin{defn}
Let $\mu,\lambda\in\mathbb{P}^{+}$, and label the $\times$'s in
the diagram $D_{\mu}$ from left to right with $1,\ldots,r$. A \emph{right
path} from $D_{\mu}$ to $D_{\lambda}$ is a sequence of right moves
$\theta=R_{i_{1}}\circ\dots\circ R_{i_{k}}$ where $i_{1}\le\ldots\le i_{k}$
and $\theta(D_{\mu})=D_{\lambda}$. The length of the path is $l\left(\theta\right):=k$.\end{defn}
\begin{example}
Let $D_{\mu}$ be
\[
\ldots\stackrel{}{-1}\quad\stackrel{}{0}\quad\stackrel{\times}{1}\quad\stackrel{<}{2}\quad\stackrel{}{3}\quad\stackrel{\boxed{\times}}{4}\quad\stackrel{>}{5}\quad\stackrel{\boxed{\times}}{6}\quad\stackrel{}{7}\quad\stackrel{\boxed{{\color{white}\square}}}{8}\quad\stackrel{>}{9}\quad\stackrel{>}{10}\quad\stackrel{}{11}\quad\stackrel{}{12}\ldots
\]
and $D_{\lambda}$ be as in Example \ref{exp: lambda diagram}. The
boxes in the diagram represent the locations of the $\times$'s in
$\lambda$. Then there are two paths from $D_{\mu}$ to $D_{\lambda}$,
namely 
\begin{eqnarray*}
D_{\lambda} & = & R_{1}\circ R_{1}\circ R_{2}\circ R_{3}\circ R_{3}\left(D_{\mu}\right)\\
D_{\lambda} & = & R_{1}\circ R_{1}\circ R_{2}\left(D_{\mu}\right).
\end{eqnarray*}
The first path sends the $i$-th $\times$ of $\mu$ to the $i$-th
box whereas the second path permutes the order. Not all such permutations
are valid, for example the third $\times$ in $\mu$ can not be moved
to the left. Also suppose the second and third $\times$'s will remain
in the first and second boxes, respectively. Then the first $\times$
will never reach the box at $8$ since it is marked by the $\times$
at $4$, and hence such a path would be invalid. 
\end{example}
If there exist paths from $D_{\mu}$ to $D_{\lambda}$, then one of
them sends the $i$-th $\times$ of $\mu^{\rho}$ to the location
of the $i$-th $\times$ of $\lambda^{\rho}$. This path is unique
because the $\times$'s are moved in order. We call it the \emph{trivial
path }from $D_{\mu}$ to $D_{\lambda}$ and denote its length by $l_{\lambda,\mu}$
(this was denoted as $l\left(\lambda,\mu\right)$ in \cite[(3.15)]{SZ1}).

The trivial path is strictly longer than the rest of the paths. Indeed,
in other paths, there is at least one $\times$ which is not moved
as far as possible. This implies that another $\times$ will jump
over it, making the move longer. So in this case one needs less moves
to fill all of the boxes. 

By \cite[Corollary 3.39]{B}, we have 
\[
K_{\lambda,\mu}\left(q\right)=\sum_{\theta\in P}q^{l\left(\theta\right)}
\]
where $P$ is the set of paths from $D_{\mu}$ to $D_{\lambda}$.
In the previous example $K_{\lambda,\mu}=q^{5}+q^{3}$.

\subsection{Kazhdan-Lusztig polynomials and the Su-Zhang character formula}

First we show that $\lambda$ being totally connected is equivalent
to all paths to $D_{\lambda}$ being trivial. This yields a new characterization
of KW-modules in terms of the Kazhdan-Lusztig polynomials. Then we
use Brundan's algorithm to give a closed formula for $e^{\rho}R\cdot\mbox{ch }L_{\pi_{st}}\left(\lambda\right)$,
which was originally proven in \cite[4.13]{SZ1}.
\begin{lem}
Let $\lambda\in\mathbb{P}^{+}$. Then $\lambda$ is totally connected
if and only if for every $\mu\in\mathbb{P}^{+}$, there is at most
one path from $D_{\mu}$ to $D_{\lambda}$.\end{lem}
\begin{proof}
We will refer to the locations of the $\times$'s of the weight diagram
$D_{\lambda}$ as boxes and the paths will send the $\times$'s in
$D_{\mu}$ to boxes.

The weight $\lambda$ is totally connected when there are no empty
spots between the boxes. This implies that a path from $D_{\mu}$
to $D_{\lambda}$ must send the $i$-th $\times$ of $\mu$ to the
$i$-th box. Indeed, if the $i$-th $\times$ of $\mu$ is sent to
the $j$-th box, $j<i$, then the next empty box to the right of the
$j$-th box will be marked by the $i$-th $\times$, and so no other
$\times$ can be sent there. This implies that the path must be unique. 

Suppose that there exists $\mu$ for which there is a non trivial
path to $D_{\lambda}$. In this path there is an $\times$ of $\mu$,
say the $i$-th, which is sent to the $j$-th box where $j<i$. Then
the next empty spot after this box, can not have a box in it. So $\lambda$
is not totally connected.\end{proof}
\begin{cor}
\label{cor:kl polys for tame}A module $L_{\pi_{st}}\left(\lambda\right)$
is  a KW-module if and only if all its Kazhdan-Lusztig polynomials
are monomials. In this case, $K_{\lambda,\mu}\left(q\right)=q^{l_{\lambda,\mu}}$.
\end{cor}

From (\ref{eq:composition factors of Kac modules}) and (\ref{eq:Kac character})
we also obtain the following, which is a special case of \cite[Thm. 4.1]{SZ1}.
\begin{cor}
Suppose $\lambda$ is totally connected. Since the unique path from
$D_{\mu}$ to $D_{\lambda}$ is the trivial one, we get 
\[
e^{\rho}R\cdot\mbox{ch }L_{\pi_{st}}\left(\lambda\right)=\sum_{\mu\in P_{\lambda}}\left(-1\right)^{l_{\lambda,\mu}}\mathcal{F}_{W}\left(e^{\mu^{\rho}}\right)
\]
where $P_{\lambda}\subset\mathbb{P}^{+}$ is the set of $\mu\in\mathbb{P}^{+}$
for which there is a path from $D_{\mu}$ to $D_{\lambda}$. 
\end{cor}
Note that for each $\mu\in P_{\lambda}$, the $W$ orbit of $\mu^{\rho}$
intersects $\left(\lambda^{\rho}-\mathbb{N}S_{\lambda}\right)$. We
denote by $\overline{\mu}$ the unique maximal element of this intersection
with respect to the standard order on $\mathfrak{h}^{*}$. We define
\[
C_{\lambda}^{\mathrm{Lexi}}:=\left\{ \overline{\mu}\ \mid\ \mu\in P_{\lambda}\right\} .
\]
Since $P_{\lambda}\subset\mathbb{P}^{+}$, this defines a bijection
between the sets $P_{\lambda}$ and $C_{\lambda}^{\mathrm{Lexi}}$. 

For $\mu\in P_{\lambda}$, we can realize $\overline{\mu}$ more explicitly
as follows. If $\varepsilon_{i}$ and $\delta_{j}$ are typical nodes
of $\lambda^{\rho}$, then $\left(\overline{\mu},\varepsilon_{i}\right):=\left(\lambda^{\rho},\varepsilon_{i}\right)$
and $\left(\overline{\mu},\delta_{j}\right):=\left(\lambda^{\rho},\delta_{j}\right)$.
The location of the atypical entries for $\overline{\mu}$ is determined
by locations in $\lambda^{\rho}$. The set of atypical entries of
$\overline{\mu}$ correspond to the set of atypical entries of $\mu^{\rho}$,
ordered such that the $\varepsilon$-atypical entries are decreasing
and the $\delta$-atypical entries are increasing. 
\begin{example}
Consider 
\begin{eqnarray*}
\lambda^{\rho} & = & 10\varepsilon_{1}+9\varepsilon_{2}+\underline{8\varepsilon_{3}}+\underline{6\varepsilon_{4}}+5\varepsilon_{5}+\underline{4\varepsilon_{6}}-2\delta_{1}-\underline{4\delta_{2}}-\underline{6\delta_{3}}-\underline{8\delta_{4}}\\
\mu^{\rho} & = & 10\varepsilon_{1}+9\varepsilon_{2}+\underline{6\varepsilon_{3}}+5\varepsilon_{4}+\underline{4\varepsilon_{5}}+\underline{\varepsilon_{6}}-\underline{\delta_{1}}-2\delta_{2}-\underline{4\delta_{3}}-\underline{6\delta_{4}}\\
\overline{\mu} & = & 10\varepsilon_{1}+9\varepsilon_{2}+\underline{6\varepsilon_{3}}+\underline{4\varepsilon_{4}}+5\varepsilon_{5}+\underline{\varepsilon_{6}}-2\delta_{1}-\underline{\delta_{2}}-\underline{4\delta_{3}}-\underline{6\delta_{4}}.
\end{eqnarray*}
Here $S_{\lambda}=\left\{ \varepsilon_{6}-\delta_{2},\varepsilon_{4}-\delta_{3},\varepsilon_{3}-\delta_{4}\right\} $
and $\lambda^{\rho}=\overline{\mu}+3\left(\varepsilon_{6}-\delta_{2}\right)+2\left(\varepsilon_{4}-\delta_{3}\right)+2\left(\varepsilon_{3}-\delta_{4}\right)$.\end{example}
\begin{rem}
\label{Rem: the length of w s.t. w(mu)=00003Dbar{mu}}The element
$w\in W$ for which $w(\mu^{\rho})=\overline{\mu}$ can be described
explicitly in terms of the trivial path $\theta$. Denote $\theta=R_{i_{1}}\circ\dots\circ R_{i_{N}}$,
then $w=w_{1}\cdot\dots\cdot w_{N}$ where each $w_{j}$ is defined
as follows. Suppose that the move $R_{i_{j}}$ moved the $\times$
at $n_{j}$ to an empty spot at $n_{j}+k_{j}$, namely, it skipped
over $k_{j}-1$ atypical spots with $>$'s and $<$'s. Then $w_{j}=s_{1}\cdot\dots\cdot s_{k_{j}-1}$
where $s_{i}$ is of the form $s_{\varepsilon_{l}-\varepsilon_{l+1}}$
if the $i$-th skip is over the $>$ of \textrm{$\varepsilon_{l}$}
and is of the form $s_{\delta_{l}-\delta_{l+1}}$ if it is over the
$<$ of $\delta_{l}$. In particular, $l\left(w_{j}\right)=k_{j}-1$.
Moreover $l\left(w\right)=\sum l\left(w_{i}\right)$. 
\end{rem}
The following lemma is the main step of the proof in which we move
from an algorithmic formula to a closed one, and it is a special case
of \cite[Thm. 4.2]{SZ1}.
\begin{lem}
One has\label{lem:Lexi sum} 
\[
e^{\rho}R\cdot\mbox{ch }L_{\pi_{st}}\left(\lambda\right)=\sum_{\overline{\mu}\ \in\ C_{\lambda}^{\mathrm{Lexi}}}\left(-1\right)^{\left|\lambda^{\rho}-\overline{\mu}\right|_{S_{\lambda}}}\mathcal{F}_{W}\left(e^{\overline{\mu}}\right).
\]
\end{lem}
\begin{proof}
Let us show that for each $\mu\in P_{\lambda}$, 
\begin{equation}
\left(-1\right)^{l_{\lambda,\mu}}\mathcal{F}_{W}\left(e^{\mu^{\rho}}\right)=\left(-1\right)^{\left|\lambda^{\rho}-\overline{\mu}\right|_{S_{\lambda}}}\mathcal{F}_{W}\left(e^{\overline{\mu}}\right).\label{eq:path to hight}
\end{equation}
Let $w\in W$ such that $w(\mu^{\rho})=\overline{\mu}$. We claim
that $\left|\lambda^{\rho}-\overline{\mu}\right|_{S_{\lambda}}=l_{\lambda,\mu}+l\left(w\right)$
which proves (\ref{eq:path to hight}). Indeed, the number $\left|\lambda^{\rho}-\overline{\mu}\right|_{S_{\lambda}}$
is the sum of the differences between the atypical entries of $\lambda^{\rho}$
and $\overline{\mu}$. This is equal to the number of moves in the
path $l_{\lambda,\mu}$ plus the number of spots being skipped. By
Remark \ref{Rem: the length of w s.t. w(mu)=00003Dbar{mu}}, $l\left(w\right)$
is exactly the number of atypical spots skipped.

\end{proof}

\begin{proof}[Proof of Theorem \ref{thm:SZ arrow}]
\emph{}Our proof goes as follows. First we express $C_{\lambda}^{\mathrm{Lexi}}$
in terms of $(\lambda_{st}^{\rho})^{\Uparrow}$. Then we add more
summands to the expression \textrm{
\[
\sum_{\overline{\mu}\ \in\ C_{\lambda}^{\mathrm{Lexi}}}\left(-1\right)^{\left|\lambda^{\rho}-\overline{\mu}\right|_{S_{\lambda}}}\mathcal{F}_{W}(e^{\overline{\mu}})
\]
} which are annihilated by $\mathcal{F}_{W}$ but that allow us to
write the sum in a nicer way. 

Denote the $\lambda^{\rho}$-maximal atypical set by $S_{\lambda}=\left\{ \beta_{1},\ldots,\beta_{r}\right\} $,
where the elements $\beta_{i}=\varepsilon_{s_{i}}-\delta_{t_{i}}$
are ordered such that $t_{i}<t_{i+1}$. Let $k_{1},\ldots,k_{r}\in\mathbb{N}$
be such that $\lambda^{\rho}=(\lambda^{\rho})^{\Uparrow}-\sum_{i=1}^{r}k_{i}\beta_{i}$.
Then $k_{1}>\cdots>k_{r}=0$. Let $a_{1}<\ldots<a_{r}$ be the atypical
entries of $\lambda^{\rho}$. Then $a_{i}+k_{i}=a_{j}+k_{j}$ and
we have 
\begin{eqnarray*}
C_{\lambda}^{\mathrm{Lexi}} & = & \left\{ \lambda^{\rho}-\sum_{i=1}^{r}n_{i}\beta_{i}\ \mid\ n_{i}\in\mathbb{N},\ a_{1}-n_{1}<a_{2}-n_{2}<\ldots<a_{r}-n_{r}\right\} \\
 & = & \left\{ \lambda^{\rho}-\sum_{i=1}^{r}n_{i}\beta_{i}\ \mid\ n_{i}\in\mathbb{N},\ k_{1}+n_{1}>k_{2}+n_{2}>\ldots>k_{r}+n_{r}\right\} \\
 & = & \left\{ (\lambda^{\rho})^{\Uparrow}-\sum_{i=1}^{r}\left(k_{i}+n_{i}\right)\beta_{i}\ \mid\ n_{i}\in\mathbb{N},\ k_{1}+n_{1}>k_{2}+n_{2}>\ldots>k_{r}+n_{r}\right\} \\
 & = & \left\{ (\lambda^{\rho})^{\Uparrow}-\sum_{i=1}^{r}m_{i}\beta_{i}\ \mid\ m_{i}\in\mathbb{Z}_{\ge k_{i}},\ m_{1}>m_{2}>\ldots>m_{r}\right\} .
\end{eqnarray*}
Now let us enlarge $C_{\lambda}^{\mathrm{Lexi}}$, namely, we define
\[
\overline{C_{\lambda}^{\mathrm{Lexi}}}=\left\{ (\lambda^{\rho})^{\Uparrow}-\sum_{i=1}^{r}m_{i}\beta_{i}\ \mid\ m_{i}\in\mathbb{N},\ m_{1}>m_{2}>\ldots>m_{r}\right\} 
\]
 and show that \textrm{
\begin{equation}
\sum_{\overline{\mu}\ \in\ C_{\lambda}^{\mathrm{Lexi}}}\left(-1\right)^{\left|\lambda^{\rho}-\overline{\mu}\right|_{S_{\lambda}}}\mathcal{F}_{W}(e^{\overline{\mu}})=\sum_{\nu\ \in\ \overline{C_{\lambda}^{\mathrm{Lexi}}}}\left(-1\right)^{\left|\lambda^{\rho}-\nu\right|_{S_{\lambda}}}\mathcal{F}_{W}(e^{\nu})\label{eq:from lexi to lexi bar}
\end{equation}
} In particular, we claim that that $\mathcal{F}_{W}\left(e^{\nu}\right)=0$
if $\nu$ is of the form $\nu=(\lambda^{\rho})^{\Uparrow}-\sum_{i=1}^{r}m_{i}\beta_{i}\in\overline{C_{\lambda}^{\mathrm{Lexi}}}$
and $m_{i}<k_{i}$ for some $1\le i\le r$. Indeed, let $j$ be such
that $m_{j}<k_{j}$ and $m_{i}\ge k_{i}$ for all $i>j$. Note that
since $\lambda$ is totally connected, all the integers between $a_{r}$
and $a_{j}+1$ are entries of $\lambda^{\rho}$. The typical entries
of $\nu$ are the same as of $\lambda^{\rho}$ and there are $r-j+1$
atypical entries which are strictly greater than $a_{j}$. This implies
that there must be equal entries of the same type, that is $\nu$
has a stabilizer in $W$. Hence, by Lemma \ref{lem:reflection yields F=00003D0}
and Lemma \ref{lem:stabilizer has a reflection} we conclude that
$\mathcal{F}_{W}\left(e^{\nu}\right)=0$.

Let $W_{r}$ be the subgroup of $W$ that permutes $S_{\lambda}$.
This subgroup is generated by elements of the form $s{}_{\varepsilon_{i}-\varepsilon_{j}}s_{\delta_{i'}-\delta_{j'}}$
where $\varepsilon_{i}-\delta_{i'},\varepsilon_{j}-\delta_{j'}\in S_{\lambda}$
so $\left|W_{r}\right|=r!$ and all $w\in W_{r}$ have positive sign.
Hence 
\[
\mathcal{F}_{W}\left(\sum_{w\in W_{r}}we^{\nu}\right)=r!\mathcal{F}_{W}\left(e^{\nu}\right)
\]
for any $\nu\in\mathfrak{h}^{*}$. Let $W_{r}(\overline{C_{\lambda}^{\mathrm{Lexi}})}=\{w(\nu)\mid w\in W_{r},\ \nu\in\overline{C_{\lambda}^{\mathrm{Lexi}}}\}$.
Then 
\[
W_{r}(\overline{C_{\lambda}^{\mathrm{Lexi}})}=\left\{ (\lambda^{\rho})^{\Uparrow}-\sum_{i=1}^{r}m_{i}\beta_{i}\ \mid\ m_{i}\in\mathbb{N},\ m_{i}\neq m_{j}\text{ for }i\neq j\right\} ,
\]
and so elements from $((\lambda^{\rho})^{\Uparrow}-\mathbb{N}S_{\lambda})\setminus W_{r}(\overline{C_{\lambda}^{\mathrm{Lexi}}})$
have a stabilizer in $W$. Thus,
\begin{eqnarray*}
r!(-1)^{\left|(\lambda^{\rho})^{\Uparrow}-\lambda^{\rho}\right|_{S_{\lambda}}}\sum_{\nu\ \in\ \overline{C_{\lambda}^{\mathrm{Lexi}}}}\left(-1\right)^{\left|\lambda^{\rho}-\nu\right|_{S_{\lambda}}}\mathcal{F}_{W}\left(e^{\nu}\right) & = & \sum_{\nu\ \in\ \overline{C_{\lambda}^{\mathrm{Lexi}}}}\left(-1\right)^{\left|(\lambda^{\rho})^{\Uparrow}-\nu\right|_{S_{\lambda}}}\mathcal{F}_{W}\left(\sum_{w\in W_{r}}e^{w(\nu)}\right)\\
 & = & \sum_{\nu\ \in\ (\lambda^{\rho})^{\Uparrow}-\mathbb{N}S_{\lambda}}\left(-1\right)^{\left|(\lambda^{\rho})^{\Uparrow}-\nu\right|_{S_{\lambda}}}\mathcal{F}_{W}\left(e^{\nu}\right)\\
 & = & \mathcal{F}_{W}\left(\sum_{\nu\ \in\ (\lambda^{\rho})^{\Uparrow}-\mathbb{N}S_{\lambda}}\left(-1\right)^{\left|(\lambda^{\rho})^{\Uparrow}-\nu\right|_{S_{\lambda}}}e^{\nu}\right)\\
 & = & \mathcal{F}_{W}\left(\frac{e^{(\lambda^{\rho})^{\Uparrow}}}{\prod_{\beta\in S_{\lambda}}\left(1+e^{-\beta}\right)}\right).
\end{eqnarray*}

\end{proof}
The character formula in the following theorem is motivated by the
denominator identity given in \cite[(1.10)]{GKMP} for $\pi_{st}$,
and can be proven using Lemma \ref{lem:Lexi sum}, Formula (\ref{eq:from lexi to lexi bar})
and the methods above.
\begin{thm}
\label{thm:other character formula in standard}Let $\lambda_{st}$
be a totally connected weight with a $\lambda^{\rho}$-maximal isotropic
set $\beta_{1},\dots,\beta_{r}$ ordered such that $\beta_{i}<\beta_{i+1}$
for $i=1,\dots,r-1$. Then
\[
e^{\rho}R\cdot\mbox{ch }L_{\pi_{st}}\left(\lambda\right)=\mathcal{F}_{W}\left(\frac{e^{(\lambda^{\rho})^{\Uparrow}}}{\left(1+e^{-\beta_{1}}\right)\left(1-e^{-\beta_{1}-\beta_{2}}\right)\ldots\left(1-\left(-1\right)^{r}e^{-\sum_{i=1}^{r}\beta_{i}}\right)}\right).
\]

\end{thm}

\section{Kac-Wakimoto character formula for  KW-modules}

\subsection{The special case.}

Let us show that Theorem \ref{thm:SZ arrow} generalizes to other
sets of simple roots by proving that the character formula is preserved
under the steps of the shortening algorithm given in the proof of
Theorem \ref{thm:tame equivalent totally connected}. This will prove
the Kac-Wakimoto character formula for the special set of simple roots.

For a totally connected highest weight $\lambda$ of a finite dimensional
simple module $L_{\pi_{st}}(\lambda)$, we let $\pi_{k}$ denote the
set of simple roots obtained after $k$ steps of the shortening algorithm
applied to $\lambda^{\rho}$. Let $\lambda_{\pi_{k}}\in\mathfrak{h}^{*}$
be such that $L=L_{\pi_{k}}\left(\lambda_{\pi_{k}}\right)$. Set $S_{0}=S_{\lambda}$
and let $S_{k}$ be the $\lambda_{\pi_{k}}^{\rho}$-maximal isotropic
set corresponding to the arc arrangement obtained by the $k$-th step
of the algorithm. Then we have the following:
\begin{thm}
\label{thm:arrow formula for taming steps}Let $\lambda$ be a totally
connected weight and let $L=L_{\pi_{st}}\left(\lambda\right)$. Then
\begin{equation}
e^{\rho}R\cdot\mbox{ch }L=\frac{\left(-1\right)^{\left|\left(\lambda_{\pi_{k}}^{\rho}\right)^{\Uparrow}-\lambda_{\pi_{k}}^{\rho}\right|_{S_{k}}}}{r!}\mathcal{F}_{W}\left(\frac{e^{\left(\lambda_{\pi_{k}}^{\rho}\right)^{\Uparrow}}}{\prod_{\beta\in S_{k}}\left(1+e^{-\beta}\right)}\right).\label{eq:arrow formula k step}
\end{equation}
 
\end{thm}
Our proof is by induction on the steps of the shortening algorithm.
Let us first see an example.
\begin{example}
Given $\lambda_{st}^{\rho}$ corresponding to Diagram (\ref{diagram example in standard choice})
from Example \ref{exp:taming}, we show that Formula (\ref{eq:arrow formula k step})
holds after one step of the shortening algorithm.

To obtain $\left(\lambda_{st}^{\rho}\right)^{\Uparrow}=\lambda_{st}^{\rho}+2\left(\varepsilon_{4}-\delta_{2}\right)$
from the $\lambda_{st}^{\rho}$ diagram, each entry labeled with a
$5$ should be replaced by a $7$. To start the algorithm, we first
push the entries $1$ and $2$ outside of the arcs, which does not
change $\lambda_{st}^{\rho}$ and $S_{\lambda_{st}}$ so the formula
is clearly preserved. 

Next we apply $r_{\varepsilon_{4}-\delta_{2}}$ and obtain $\lambda_{\pi_{1}}^{\rho}$
corresponding to Diagram (\ref{diagram example after reflecting}).
Then $\lambda_{\pi_{1}}^{\rho}=\lambda_{st}^{\rho}+\left(\varepsilon_{4}-\delta_{2}\right)$
and $S_{1}=s_{\varepsilon_{3}-\varepsilon_{4}}S_{\lambda_{st}}$ .
So $\left(\lambda_{\pi_{1}}^{\rho}\right)^{\Uparrow}=\lambda_{\pi_{1}}^{\rho}+\left(\varepsilon_{3}-\delta_{2}\right)$
and $\left(\lambda_{\pi_{1}}^{\rho}\right)^{\Uparrow}=s_{\varepsilon_{3}-\varepsilon_{4}}\left(\lambda_{st}^{\rho}\right)^{\Uparrow}$.
Hence
\begin{eqnarray*}
e^{\rho}R\cdot\mbox{ch }L & = & \frac{\left(-1\right)^{2}}{2}\mathcal{F}_{W}\left(\left(-1\right)\cdot s_{\varepsilon_{3}-\varepsilon_{4}}\frac{e^{\left(\lambda_{st}^{\rho}\right)^{\Uparrow}}}{\prod_{\beta\in S_{\lambda_{st}}}\left(1+e^{-\beta}\right)}\right)\\
 & = & \frac{\left(-1\right)^{\left|\left(\lambda_{\pi_{1}}^{\rho}\right)^{\Uparrow}-\lambda_{\pi_{1}}^{\rho}\right|_{S_{1}}}}{2}\mathcal{F}_{W}\left(\frac{e^{\left(\lambda_{\pi_{1}}^{\rho}\right)^{\Uparrow}}}{\prod_{\beta\in S_{1}}\left(1+e^{-\beta}\right)}\right)
\end{eqnarray*}
and the formula is preserved. Finally, we move the $6$ out to obtain
Diagram (\ref{diagram example after one step}). Since this does not
change $\lambda_{\pi_{1}}^{\rho}$ and $S_{\pi_{1}}$ the formula
is preserved.\end{example}
\begin{proof}[Proof of Theorem \ref{thm:arrow formula for taming steps}. ]
Our proof is by induction on the steps of the shortening algorithm
from the proof of Theorem \ref{thm:tame equivalent totally connected}.
After $k$ steps, we obtain new data: $\pi_{k}$ $\lambda_{\pi_{k}}^{\rho}$,
$\lambda_{\pi_{k}}$, $S_{k}$. We express the RHS of Formula (\ref{eq:arrow formula k step})
in terms of this new data, and prove that it is equal to $e^{\rho}R\cdot\mbox{ch}L$
using the formula obtained after $k-1$ steps.

Before applying the algorithm we start by pushing the entries located
below the innermost arc outside of the arcs. Since this corresponds
to reflections with respect to roots which are not orthogonal to $\lambda_{\pi_{st}}^{\rho}$,
this changes $\pi_{st}$ but not $\lambda_{\pi_{st}}^{\rho}$ or $S_{\lambda_{st}}$.
So Formula (\ref{eq:arrow formula k step}) is unchanged.

Now suppose that the formula holds after $k-1$ steps of the algorithm,
that is,\\
\begin{equation}
e^{\rho}R\cdot\mbox{ch\ }L=\frac{\left(-1\right)^{\left|\left(\lambda_{\pi_{k-1}}^{\rho}\right)^{\Uparrow}-\lambda_{\pi_{k-1}}^{\rho}\right|_{S_{k-1}}}}{r!}\mathcal{F}_{W}\left(\frac{e^{\left(\lambda_{\pi_{k-1}}^{\rho}\right)^{\Uparrow}}}{\prod_{\beta\in S_{k-1}}\left(1+e^{-\beta}\right)}\right).\label{eq:k-1 induction step}
\end{equation}
Let us apply one more step and show that (\ref{eq:k-1 induction step})
implies (\ref{eq:arrow formula k step}). There are three cases, depending
on the location of the $b+1$ entry (see proof of Theorem \ref{thm:tame equivalent totally connected}
``$\Leftarrow$''). In each case, we reflect at all of the short arcs,
and then arrange the arcs to be of $\bullet-\times$ type. 

\begin{equation}\label{diagrams for proof of special case} \begin{aligned}
%\begin{eqnarray*}
&\begin{tikzpicture} 
%diagram
\node at (-0.75,0) {\dots};
\node at (0,0) {$\bullet$};
\node [below] at (0,0) {$\underset{\varepsilon_{m_0}}{\scriptstyle{b+1}}$};
\node at (1,0) {$\bullet$};
\node [below] at (1,0) {$\underset{\varepsilon_{m_1}}{\scriptstyle{b}}$};
\node at (2,0) {$\times$};
\node [below] at (2,0) {$\underset{\delta_{n_1}}{\scriptstyle{b}}$};

\node at (2.75,0) {\dots};
\node at (3.5,0) {$\bullet$};
\node [below] at (3.5,0) {$\underset{\varepsilon_{m_s}}{\scriptstyle{b}}$};
\node at (4.5,0) {$\times$};
\node [below] at (4.5,0) {$\underset{\delta_{n_s}}{\scriptstyle{b}}$};
\node  at (5.25,0) {\dots};

%arcs
\draw (1,0.2) arc [radius=1, start angle=120, end angle= 60];
\draw (3.5,0.2) arc [radius=1, start angle=120, end angle= 60];

%\node at (6.5,0) {$\longrightarrow$};

\end{tikzpicture}
&\longrightarrow\quad&\begin{tikzpicture}

\node at (7.3,0) {$\dots$};
\node at (8,0) {$\bullet$};
\node[below] at (8,0) {$\underset{\varepsilon_{m_0}}{\scriptstyle{b+1}}$};
\node at (9,0) {$\times$};
\node[below] at (9,0) {$\underset{\delta_{n_1}}{\scriptstyle{b+1}}$};
\node at (9.75,0) {$\dots$};
\node at (10.5,0) {$\bullet$};
\node[below] at (10.5,0) {$\underset{\varepsilon_{m_{s-1}}}{\scriptstyle{b+1}}$};
\node at (11.5,0) {$\times$};
\node[below] at (11.5,0) {$\underset{\delta_{n_s}}{\scriptstyle{b+1}}$};
\node at (12.5,0) {$\bullet$};
\node[below] at (12.5,0) {$\underset{\varepsilon_{m_s}}{\scriptstyle{b+1}}$};
\node at (13.25,0) {$\dots$};

\draw (9,0.2) arc [radius=1, start angle=60, end angle= 120];
\draw (11.5,0.2) arc [radius=1, start angle=60, end angle= 120];
\end{tikzpicture}\\
&\begin{tikzpicture} 
%diagram
\node  at (-0.7,0) {\dots};
\node at (0,0) {$\bullet$};
\node [below] at (0,0) {$\underset{\varepsilon_{m_1}}{\scriptstyle{b}}$};
\node at (1,0) {$\times$};
\node [below] at (1,0) {$\underset{\delta_{n_1}}{\scriptstyle{b}}$};
\node at (1.75,0) {$\dots$};

\node at (2.5,0) {$\bullet$};
\node [below] at (2.5,0) {$\underset{\varepsilon_{m_s}}{\scriptstyle{b}}$};
\node at (3.5,0) {$\times$};
\node [below] at (3.5,0) {$\underset{\delta_{n_s}}{\scriptstyle{b}}$};
\node at (4.5,0) {$\times$};
\node [below] at (4.5,0) {$\underset{\delta_{n_{s+1}}}{\scriptstyle{b+1}}$};
\node at (5.25,0) {$\dots$};

%arcs
\draw (0,0.2) arc [radius=1, start angle=120, end angle= 60];
\draw (2.5,0.2) arc [radius=1, start angle=120, end angle= 60];

\end{tikzpicture}
&\longrightarrow\quad&\begin{tikzpicture}

%\node at (6.5,0) {$\longrightarrow$};
\node at (7.3,0) {$\dots$};
\node at (8,0) {$\times$};
\node[below] at (8,0) {$\underset{\delta_{n_1}}{\scriptstyle{b+1}}$};
\node at (9,0) {$\bullet$};
\node[below] at (9,0) {$\underset{\varepsilon_{m_1}}{\scriptstyle{b+1}}$};
\node at (10,0) {$\times$};
\node[below] at (10,0) {$\underset{\delta_{n_2}}{\scriptstyle{b+1}}$};
\node at (10.75,0) {$\dots$};
\node at (11.5,0) {$\bullet$};
\node[below] at (11.5,0) {$\underset{\varepsilon_{m_s}}{\scriptstyle{b+1}}$};
\node at (12.5,0) {$\times$};
\node[below] at (12.5,0) {$\underset{\delta_{n_{s+1}}}{\scriptstyle{b+1}}$};
\node at (13.25,0) {$\dots$};

\draw (10,0.2) arc [radius=1, start angle=60, end angle= 120];
\draw (12.5,0.2) arc [radius=1, start angle=60, end angle= 120];
\end{tikzpicture}\\
&\begin{tikzpicture}
%diagram
\node at (-0.7,0) {\dots};
\node at (0,0) {$\bullet$};
\node [below] at (0,0) {$\underset{\varepsilon_{m_0}}{\scriptstyle{b+1}}$};
\node at (1,0) {$\bullet$};
\node [below] at (1,0) {$\underset{\varepsilon_{m_1}}{\scriptstyle{b}}$};
\node at (2,0) {$\times$};
\node [below] at (2,0) {$\underset{\delta_{n_1}}{\scriptstyle{b}}$};

\node at (3,0) {\dots};
\node at (4,0) {$\bullet$};
\node [below] at (4,0) {$\underset{\varepsilon_{m_s}}{\scriptstyle{b}}$};
\node at (5,0) {$\times$};
\node [below] at (5,0) {$\underset{\delta_{n_s}}{\scriptstyle{b}}$};
\node at (6,0) {$\times$};
\node [below] at (6,0) {$\underset{\delta_{n_{s+1}}}{\scriptstyle{b+1}}$};
\node at (6.7,0) {\dots};

%arcs
\draw (1,0.2) arc [radius=1, start angle=120, end angle= 60];
\draw (4,0.2) arc [radius=1, start angle=120, end angle= 60];
\draw (0,0.2) arc [radius=8.8, start angle=110, end angle= 70];

%\node at (7.5,0) {$\longrightarrow$};
\end{tikzpicture}
&\longrightarrow\quad&\begin{tikzpicture}
\node at (8.3,0) {$\dots$};
\node at (9,0) {$\bullet$};
\node[below] at (9,0) {$\underset{\varepsilon_{m_0}}{\scriptstyle{b+1}}$};
\node at (10,0) {$\times$};
\node[below] at (10,0) {$\underset{\delta_{n_{1}}}{\scriptstyle{b+1}}$};
\node at (11,0) {$\dots$};
\node at (12,0) {$\bullet$};
\node[below] at (12,0) {$\underset{\varepsilon_{m_s}}{\scriptstyle{b+1}}$};
\node at (13,0) {$\times$};
\node[below] at (13,0) {$\underset{\delta_{n_{s+1}}}{\scriptstyle{b+1}}$};
\node at (13.7,0) {$\dots$};

\draw (10,0.2) arc [radius=1, start angle=60, end angle= 120];
\draw (13,0.2) arc [radius=1, start angle=60, end angle= 120];
\end{tikzpicture}
%\end{eqnarray*}
\end{aligned}\end{equation}
\\
This does not change the typical entries or the maximal atypical entry,
so $\left(\lambda_{\pi_{k-1}}^{\rho}\right)^{\Uparrow}$ only differs
$\left(\lambda_{\pi_{k}}^{\rho}\right)^{\Uparrow}$ by a permutation
of the entries. Moreover, $\lambda_{\pi_{k}}^{\rho}$ is ``closer''
to $\left(\lambda_{st}^{\rho}\right)^{\Uparrow}$ than in the previous
step because some entries were increased. Indeed, the sequence of
odd reflections is with respect to the simple atypical odd roots $S_{k-1}\cap\pi_{k-1}$,
so we have 
\[
\lambda_{\pi_{k}}^{\rho}=\lambda_{\pi_{k-1}}^{\rho}+\sum_{\beta\in S_{k-1}\cap\pi_{k-1}}\beta.
\]
For each diagram from (\ref{diagrams for proof of special case}),
there exists $w\in W$ such that 
\[
wS_{k-1}=S_{k},\qquad w\left(\lambda_{\pi_{k-1}}^{\rho}\right)^{\Uparrow}=\left(\lambda_{\pi_{k}}^{\rho}\right)^{\Uparrow},\quad\text{{and}}\qquad l\left(w\right)=\left|S_{k-1}\cap\pi_{k-1}\right|.
\]
 Indeed, we have that $S_{k-1}\cap\pi_{k-1}={\{\beta_{1},\ldots,\beta_{s}\}}$,
where each $\beta_{i}=\varepsilon_{m_{i}}-\delta_{n_{i}}$ is a simple
atypical root corresponding to a diagram of (\ref{diagrams for proof of special case}).
For the first and third case, we take $w\in Sym_{m}$ such that $w\left(\varepsilon_{m_{i}}\right)=\varepsilon_{m_{i-1}}$,
$w\left(\varepsilon_{m_{0}}\right)=\varepsilon_{m_{s}}$ and all other
elements are fixed, and in the second case, we take $w\in Sym_{n}$
such that $w(\delta_{i})=\delta_{i+1}$, $w(\delta_{s+1})=\delta_{1}$
and all other elements are fixed. Then
\begin{eqnarray*}
e^{\rho}R\cdot\mbox{ch }L & = & \frac{\left(-1\right)^{\left|\left(\lambda_{\pi_{k-1}}^{\rho}\right)^{\Uparrow}-\lambda_{\pi_{k-1}}^{\rho}\right|_{S_{k-1}}}}{r!}\mathcal{F}_{W}\left(\frac{e^{\left(\lambda_{\pi_{k-1}}^{\rho}\right)^{\Uparrow}}}{\prod_{\beta\in S_{k-1}}\left(1+e^{-\beta}\right)}\right)\\
 & = & \frac{\left(-1\right)^{\left|\left(\lambda_{\pi_{k-1}}^{\rho}\right)^{\Uparrow}-\lambda_{\pi_{k-1}}^{\rho}\right|_{S_{k-1}}+\left|S_{\lambda_{k-1}}\cap\pi_{k-1}\right|}}{r!}\mathcal{F}_{W}\left(\left(-1\right)^{l\left(w\right)}w\frac{e^{\left(\lambda_{\pi_{k-1}}^{\rho}\right)^{\Uparrow}}}{\prod_{\beta\in S_{k-1}}\left(1+e^{-\beta}\right)}\right)\\
 & = & \frac{\left(-1\right)^{\left|\left(\lambda_{\pi_{k}}^{\rho}\right)^{\Uparrow}-\lambda_{\pi_{k}}^{\rho}\right|_{S_{k}}}}{r!}\mathcal{F}_{W}\left(\frac{e^{\left(\lambda_{\pi_{k}}^{\rho}\right)^{\Uparrow}}}{\prod_{\beta\in S_{k}}\left(1+e^{-\beta}\right)}\right).
\end{eqnarray*}

 Finally, we push the smallest unmatched entry outside of the arcs.
Since this corresponds to reflections with respect to roots which
are not orthogonal to $\lambda_{\pi_{k-1}}^{\rho}$, this does not
change formula (\ref{eq:arrow formula k step}).
\end{proof}
Since for the special set of simple roots all atypical entries are
equal, we conclude the following.
\begin{cor}
\label{cor:KW special tame}Let $L$ be a  KW-module with an admissible
choice of simple roots $\pi$, a corresponding highest weight $\lambda$
and a $\lambda^{\rho}$-maximal isotropic subset $S_{\lambda}\subset\pi$,
and let $r=|S_{\lambda}|$. If $\pi$ is the special set, then

\begin{equation}
e^{\rho}R\cdot\mbox{ch }L_{\pi}\left(\lambda\right)=\frac{1}{r!}\mathcal{F}_{W}\left(\frac{e^{\lambda^{\rho}}}{\prod_{\beta\in S_{\lambda}}\left(1+e^{-\beta}\right)}\right).\label{eq:KW}
\end{equation}

\end{cor}

\subsection{The general case.}

Let us prove the Kac-Wakimoto character formula for an arbitrary admissible
choice of simple roots by showing that we can move from any admissible
choice of simple roots to the special set in a way that preserves
the formula in (\ref{eq:KW}).

We will use the following four moves on arc diagrams.

\begin{enumerate}
\item\qquad\qquad\qquad
$\begin{array}{c}\begin{tikzpicture} 
\node  at (1,0) {$\dots$}; 
\node  at (2,0) {$\bullet$}; 
\node [below] at (2,0) {$\underset{\varepsilon_{i}}{\scriptstyle{a}}$}; 
\node  at (3,0) {$\times$}; 
\node [below] at (3,0) {$\underset{\delta_{j}}{\scriptstyle{b}}$}; 
\node  at (4,0) {$\dots$}; 
\node [below] at (5,0) {$\leftrightarrow$}; 
\node  at (6,0) {$\dots$}; 
\node  at (7,0) {$\times$}; 
\node [below] at (7,0) {$\underset{\delta_{j}}{\scriptstyle{b}}$}; 
\node  at (8,0) {$\bullet$}; 
\node [below] at (8,0) {$\underset{\varepsilon_{i}}{\scriptstyle{a}}$}; 
\node  at (9,0) {$\dots$}; 
%arcs 
\end{tikzpicture}\end{array}$\label{movediagram1} \\
 \item\qquad\qquad\qquad $\begin{array}{c}
\begin{tikzpicture}  \node  at (1,0) {$\dots$}; \node  at (2,0) {$\bullet$}; \node [below] at (2,0) {$\underset{\varepsilon_{i}}{\scriptstyle{a}}$}; \node  at (3,0) {$\times$}; \node [below] at (3,0) {$\underset{\delta_{j}}{\scriptstyle{a}}$}; \node  at (4,0) {$\dots$}; \node [below] at (5,0) {$\leftrightarrow$}; \node  at (6,0) {$\dots$}; \node  at (7,0) {$\times$}; \node [below] at (7,0) {$\underset{\delta_{j}}{\scriptstyle{a+1}}$}; \node  at (8,0) {$\bullet$}; \node [below] at (8,0) {$\underset{\varepsilon_{i}}{\scriptstyle{a+1}}$}; \node  at (9,0) {$\dots$};
\draw (3,0.2) arc [radius=1, start angle=60, end angle= 120]; \draw (8,0.2) arc [radius=1, start angle=60, end angle= 120]; \end{tikzpicture}\end{array}$\label{movediagram2}\\
\item\qquad\qquad $\begin{array}{c}\begin{tikzpicture} 
\node  at (1,0) {$\dots$}; \node  at (2,0) {$\bullet$}; \node [below] at (2,0) {$\underset{\varepsilon_{i}}{\scriptstyle{a}}$}; \node  at (3,0) {$\times$}; \node [below] at (3,0) {$\underset{\delta_{j}}{\scriptstyle{a}}$}; \node  at (4,0) {$\bullet$}; \node [below] at (4,0) {$\underset{\varepsilon_{i+1}}{\scriptstyle{a}}$}; \node  at (5,0) {$\dots$}; \node [below] at (6,0) {$\leftrightarrow$}; \node  at (7,0) {$\dots$}; \node  at (8,0) {$\bullet$}; \node [below] at (8,0) {$\underset{\varepsilon_{i}}{\scriptstyle{a}}$}; \node  at (9,0) {$\bullet$}; \node [below] at (9,0) {$\underset{\varepsilon_{i+1}}{\scriptstyle{a-1}}$}; \node  at (10,0) {$\times$}; \node [below] at (10,0) {$\underset{\delta_{j}}{\scriptstyle{a-1}}$}; \node  at (11,0) {$\dots$};
\draw (3,0.2) arc [radius=1, start angle=60, end angle= 120]; \draw (10,0.2) arc [radius=1, start angle=60, end angle= 120]; \end{tikzpicture}\end{array}$\label{movediagram3}\\
\item\qquad\qquad $\begin{array}{c}\begin{tikzpicture} 
\node  at (1,0) {$\dots$}; \node  at (2,0) {$\times$}; \node [below] at (2,0) {$\underset{\delta_{j-1}}{\scriptstyle{a}}$}; \node  at (3,0) {$\bullet$}; \node [below] at (3,0) {$\underset{\varepsilon_{i}}{\scriptstyle{a}}$}; \node  at (4,0) {$\times$}; \node [below] at (4,0) {$\underset{\delta_{j}}{\scriptstyle{a}}$}; \node  at (5,0) {$\dots$}; \node [below] at (6,0) {$\leftrightarrow$}; \node  at (7,0) {$\dots$}; \node  at (8,0) {$\bullet$}; \node [below] at (8,0) {$\underset{\varepsilon_{i}}{\scriptstyle{a-1}}$}; \node  at (9,0) {$\times$}; \node [below] at (9,0) {$\underset{\delta_{j-1}}{\scriptstyle{a-1}}$}; \node  at (10,0) {$\times$}; \node [below] at (10,0) {$\underset{\delta_{j}}{\scriptstyle{a}}$}; \node  at (11,0) {$\dots$};
\draw (4,0.2) arc [radius=1, start angle=60, end angle= 120]; \draw (9,0.2) arc [radius=1, start angle=60, end angle= 120]; \end{tikzpicture}\end{array}$ \label{movediagram4}
\end{enumerate}

Each move is achieved by applying an odd reflection to the set of
simple roots $\pi\leftrightarrow\pi'$ followed by a choice of the
arc arrangement. Let us first show that each of these moves preserves
the Kac-Wakimoto character formula.
\begin{prop}
\label{prop: moves preserve formula}Suppose that $(\pi,\lambda,S_{\lambda})$
and $(\pi',\lambda',S_{\lambda'})$ differ by one of the above moves.
Then 
\[
\mathcal{F}_{W}\left(\frac{e^{\lambda^{\rho}}}{\prod_{\beta\in S_{\lambda}}(1+e^{-\beta})}\right)=\mathcal{F}_{W}\left(\frac{e^{{\lambda^{\rho}}'}}{\prod_{\beta\in S_{\lambda'}}(1+e^{-\beta})}\right).
\]
\end{prop}
\begin{proof}
We prove the claim for each move defined above. We take $\lambda^{\rho}$
to correspond to the LHS diagram and ${\lambda^{\rho}}'$ to correspond
to the RHS diagram. 

The claim is obvious for move (\ref{movediagram1}), since in this
case ${\lambda^{\rho}}'=\lambda^{\rho}$ and $S_{\lambda'}=S_{\lambda}$.

In move (\ref{movediagram2}), we have ${\lambda^{\rho}}'=\lambda^{\rho}+(\varepsilon_{i}-\delta_{j})$,
and 
\[
S_{\lambda'}=(S_{\lambda}\setminus\{\varepsilon_{i}-\delta_{j}\})\cup\{\delta_{j}-\varepsilon_{i}\},
\]
so 
\[
\begin{array}{rcl}
\medskip{\displaystyle \mathcal{F}_{W}\left(\frac{e^{{\lambda^{\rho}}'}}{\prod_{\beta\in S_{\lambda'}}(1+e^{-\beta})}\right)} & = & {\displaystyle \mathcal{F}_{W}\left(\frac{e^{{\lambda^{\rho}}+\varepsilon_{i}-\delta_{j}}}{\prod_{\beta\in S_{\lambda\setminus(\varepsilon_{i}-\delta_{j})}}(1+e^{-\beta})(1+e^{-(\delta_{j}-\varepsilon_{i})})}\right)}\\
\medskip & = & {\displaystyle \mathcal{F}_{W}\left(\frac{e^{\lambda^{\rho}}}{\prod_{\beta\in S_{\lambda}}(1+e^{-\beta})}\right),}
\end{array}
\]

as desired.
\end{proof}
In move (\ref{movediagram3}), we have ${\lambda^{\rho}}'=\lambda^{\rho}+(\delta_{j}-\varepsilon_{i+1})$,
and 
\[
S_{\lambda'}=(S_{\lambda}\setminus\{\varepsilon_{i}-\delta_{j}\})\cup\{\varepsilon_{i+1}-\delta_{j}\},
\]
so 
\[
\begin{array}{rcl}
\medskip{\displaystyle \mathcal{F}_{W}\left(\frac{e^{{\lambda^{\rho}}'}}{\prod_{\beta\in S_{\lambda'}}(1+e^{-\beta})}\right)} & = & {\displaystyle \mathcal{F}_{W}\left(\frac{e^{\lambda^{\rho}-\varepsilon_{i+1}+\delta_{j}}}{\prod_{\beta\in S_{\lambda\setminus(\varepsilon_{i}-\delta_{j})}}(1+e^{-\beta})(1+e^{-(\varepsilon_{i+1}-\delta_{j})})}\right)}\\
\medskip & = & {\displaystyle \mathcal{F}_{W}\left(s_{\varepsilon_{i}-\varepsilon_{i+1}}\left(\frac{e^{\lambda^{\rho}-\varepsilon_{i}+\delta_{j}}}{\prod_{\beta\in S_{\lambda}}(1+e^{-\beta})}\right)\right)}\\
\medskip & = & {\displaystyle -\mathcal{F}_{W}\left(\frac{e^{\lambda^{\rho}-\varepsilon_{i}+\delta_{j}}}{\prod_{\beta\in S_{\lambda}}(1+e^{-\beta})}\right)}.
\end{array}
\]
 It remains to show that 
\[
-\mathcal{F}_{W}\left(\frac{e^{\lambda^{\rho}-\varepsilon_{i}+\delta_{j}}}{\prod_{\beta\in S_{\lambda}}(1+e^{-\beta})}\right)=\mathcal{F}_{W}\left(\frac{e^{\lambda^{\rho}}}{\prod_{\beta\in S_{\lambda}}(1+e^{-\beta})}\right),
\]
which follows from the following. 
\[
\begin{array}{rcl}
\medskip{\displaystyle \mathcal{F}_{W}\left(\frac{e^{\lambda^{\rho}}}{\prod_{\beta\in S_{\lambda}}(1+e^{-\beta})}+\frac{e^{\lambda^{\rho}-\varepsilon_{i}+\delta_{j}}}{\prod_{\beta\in S_{\lambda}}(1+e^{-\beta})}\right)} & = & {\displaystyle \mathcal{F}_{W}\left(e^{\lambda^{\rho}}\frac{1+e^{-(\varepsilon_{i}-\delta_{j})}}{\prod_{\beta\in S_{\lambda}}(1+e^{-\beta})}\right)}\\
\medskip & = & {\displaystyle \mathcal{F}_{W}\left(\frac{e^{\lambda^{\rho}}}{\prod\limits _{\beta\in S_{\lambda}\setminus\{\varepsilon_{i}-\delta_{j}\}}(1+e^{-\beta})}\right)}\\
\medskip & = & 0
\end{array}
\]
The last equality holds by Lemma \ref{lem:reflection yields F=00003D0},
since the argument of $\mathcal{F}_{W}$ is preserved by the simple
reflection $s_{\varepsilon_{i}-\varepsilon_{i-1}}$, which completes
the proof for the third move.

In move (\ref{movediagram4}), we have ${\lambda^{\rho}}'=\lambda^{\rho}+(\delta_{j-1}-\varepsilon_{i})$,
and 
\[
S_{\lambda'}=(S_{\lambda}\setminus\{\varepsilon_{i}-\delta_{j}\})\cup\{\varepsilon_{i}-\delta_{j-1}\}.
\]
The proof in this case is analogous to the third move except that
the roles of the $\varepsilon$'s and $\delta$'s are interchanged. 

Recall the definition of the special set of simple roots given in
Definition \ref{def: special tame}.
\begin{prop}
\label{prop: formula for special tame to any tame}One can move from
any admissible choice of simple roots to the special set of simple
roots in a way that preserves the character formula (\ref{eq:KW}). 
\end{prop}
First we give an example.
\begin{example}
Start with the following arc diagram corresponding to an admissible
choice of simple roots.

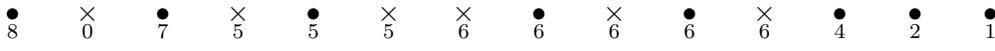
\ad{1/\cir/8,3/\cir/7,2/\x/0,4/\x/5,5/\cir/5,6/\x/5,7/\x/6,8/\cir/6,9/\x/6,10/\cir/6,11/\x/6,12/\cir/4,13/\cir/2,14/\cir/1}{5/1/60/120,
8/1/60/120, 11/1/60/120}\\
Use move (\ref{movediagram2}) twice to arrange the arcs to be of
$\bullet-\times$ type.

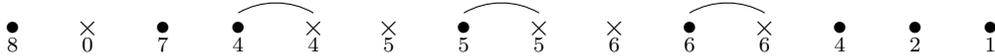
\ad{1/\cir/8,3/\cir/7,2/\x/0,4/\cir/4,5/\x/4,6/\x/5,7/\cir/5,8/\x/5,9/\x/6,10/\cir/6,11/\x/6,12/\cir/4,13/\cir/2,14/\cir/1}{5/1/60/120,
8/1/60/120, 11/1/60/120}\\
Use move (\ref{movediagram4}) three times to take all the intermediate
\x's outside the arcs.

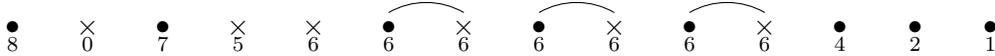
\ad{1/\cir/8,3/\cir/7,2/\x/0,4/\x/5,5/\x/6,6/\cir/6,7/\x/6,8/\cir/6,9/\x/6,10/\cir/6,11/\x/6,12/\cir/4,13/\cir/2,14/\cir/1}{7/1/60/120,
9/1/60/120, 11/1/60/120}\\
Use move (\ref{movediagram1}) on each side of the arcs to organize
the typical entries to have the \cir's precede the \x's.

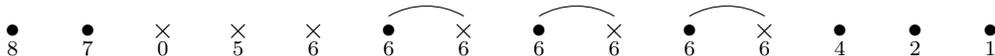
\ad{1/\cir/8,2/\cir/7,3/\x/0,4/\x/5,5/\x/6,6/\cir/6,7/\x/6,8/\cir/6,9/\x/6,10/\cir/6,11/\x/6,12/\cir/4,13/\cir/2,14/\cir/1}{7/1/60/120,
9/1/60/120, 11/1/60/120}\\
Use moves (4) and (\ref{movediagram1}) to move the typical \x-entry
6 to the right of the arcs.

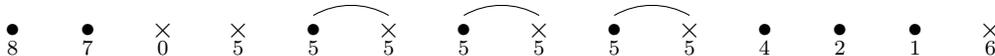
\ad{1/\cir/8,2/\cir/7,3/\x/0,4/\x/5,5/\cir/5,6/\x/5,7/\cir/5,8/\x/5,9/\cir/5,10/\x/5,14/\x/6,11/\cir/4,12/\cir/2,13/\cir/1}{6/1/60/120,
8/1/60/120, 10/1/60/120}\\
Use moves (4) and (\ref{movediagram1}) to move the typical \x-entry
5 to the right of the arcs.

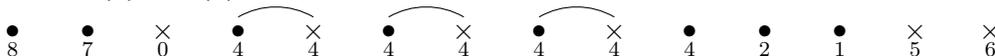
\ad{1/\cir/8,2/\cir/7,3/\x/0,4/\cir/4,5/\x/4,6/\cir/4,7/\x/4,8/\cir/4,9/\x/4,13/\x/5,14/\x/6,10/\cir/4,11/\cir/2,12/\cir/1}{5/1/60/120,
7/1/60/120, 9/1/60/120}\\
Use moves (\ref{movediagram3}) and (\ref{movediagram1}) to move
the typical \x-entry 4 to the left of the arcs.

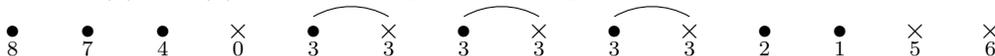
\ad{1/\cir/8,2/\cir/7,3/\cir/4,4/\x/0,5/\cir/3,6/\x/3,7/\cir/3,8/\x/3,9/\cir/3,10/\x/3,13/\x/5,14/\x/6,11/\cir/2,12/\cir/1}{6/1/60/120,
8/1/60/120, 10/1/60/120}\\
The resulting diagram is the special arc diagram. \end{example}
\begin{proof}[Proof of Proposition \ref{prop: formula for special tame to any tame}]
 We show that we can move from an arbitrary admissible choice of
simple roots to the special set in a way that uses only moves (\ref{movediagram1})
- (\ref{movediagram4}) defined above. Since each of these moves was
shown to preserve the character formula in Proposition \ref{prop: moves preserve formula},
the result then follows.

Consider an arc diagram corresponding to an arbitrary admissible choice
of simple roots. Since all arcs are short, we can use move (\ref{movediagram2})
to arrange all of the arcs to be of $\bullet-\times$ type. Then we
use move (\ref{movediagram3}) or move (\ref{movediagram4}) to place
all the arcs next to each other, which is possible due to Lemma \ref{lem:Tame structure}.
Then by Lemma \ref{lem:Tame structure}, all of the atypical entries
are equal, and we denote them by $a$. Then, on each side of the arcs,
we use move (\ref{movediagram1}) to organize the typical entries
so that the $\bullet$'s precede the $\times$'s.

Since the diagram corresponds to a finite dimensional module, we have
at each end, that the typical $\times$-entries are strictly increasing,
while the typical $\bullet$-entries are strictly decreasing (see
Remark \ref{rem: arc diagram entries}). Moreover, we can show that
on the left end, the $\times$-entries are $\leqslant a$ while the
$\bullet$-entries are $>a$. Indeed, if the biggest $\times$-entry
is greater than $a$, then the node can be moved inside the first
arc by using move (\ref{movediagram1}), and we reach a contradiction
on the value of the adjacent $\times$ entries. Also, if the smallest
$\bullet$-entry is not smaller than $a$, the node can be moved next
to the first arc by using move (\ref{movediagram1}), and the same
problem arises. Similarly, on the right end the $\times$-entries
are $>a$ while the $\bullet$-entries are $\leqslant a$. Together
with the fact that no additional arcs are possible by the maximality
property, this implies that all typical entries are distinct.

The only way that the resulting diagram is not the special arc diagram
is if one of the two inequalities above is an equality, and in particular,
if one of the typical entries equals $a$. If it is the first one,
then we use move (\ref{movediagram4}) to transfer the relevant $\times$
to the right of the arcs and then move (\ref{movediagram1}) to transfer
it to the right of all $\bullet$'s. If it is the second one, then
we use move (\ref{movediagram3}) to transfer the relevant $\bullet$
to the left of the arcs and then move (\ref{movediagram1}) to transfer
it to the left of all $\times$'s. The resulting diagram satisfies
the properties of the previous paragraph with the atypical entries
now labeled by $a-1$.

We repeat the above step until we cannot do it anymore. After each
step the atypical entries are decreased by $1$, while the labeling
set for the typical entries remains the same. Hence the process must
terminate, and the resulting diagram will be the special arc diagram.
\end{proof}
When combined with Corollary \ref{cor:KW special tame}, this concludes
the proof of the Kac-Wakimoto character formula for $\mathfrak{gl}(m|n)$
and we have the following theorem.
\begin{thm}
Let $L$ be a finite dimensional simple module. For any choice of
simple roots $\pi$ and weight $\lambda,$ such that $L=L_{\pi}(\lambda)$
and $\pi$ contains a $\lambda^{\rho}$-maximal isotropic subset $S_{\lambda}$
we have that

\[
e^{\rho}R\cdot\mbox{ch }L_{\pi}\left(\lambda\right)=\frac{1}{r!}\sum_{w\in W}(-1)^{l(w)}w\left(\frac{e^{\lambda^{\rho}}}{\prod_{\beta\in S_{\lambda}}\left(1+e^{-\beta}\right)}\right)
\]

where $r=|S_{\lambda}|$.
\end{thm}

\section{A determinantal character formula for KW-modules of $\mathfrak{gl}(m|n)$.}

In this section, we use the Kac-Wakimoto character formula to prove
a determinantal character formula for KW-modules for $\mathfrak{gl}(m|n)$,
which is motivated by the determinantal character formula proven in
\cite{MV2} for critical modules labeled by non-intersecting composite
partitions. Our determinantal character formula can be expressed using
the data of the special arc diagram for a KW-module $L$ (recall Definition
\ref{def: special tame}).

\subsection{A determinantal character formula.}

Consider the special arc diagram of a KW-module $L.$

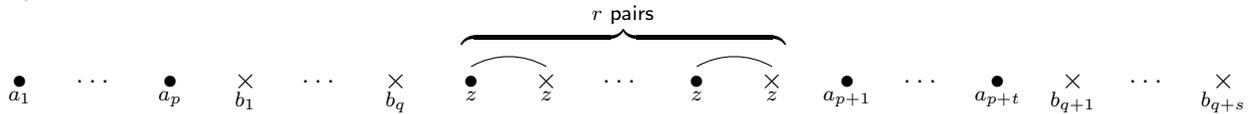
\begin{equation}\label{special tame for det formula} \hspace*{6.45cm}\raisebox{1cm}{$\stackrel{r\text{ pairs}}{\overbrace{\hspace{4.3cm}}}$} \hspace*{-10.45cm} \begin{tikzpicture} \foreach \p/\t/\w in {1/\cir/a_1,2/\dots/,3/\cir/a_p,4/\x/b_1,5/\dots/,6/\x/b_q,7/\cir/z,8/\x/z,9/\dots/,10/\cir/z,11/\x/z,12/\cir/a_{p+1},13/\dots/,14/\cir/a_{p+t}, 15/\x/b_{q+1}, 16/\dots/, 17/\x/b_{q+s}}  {   \node [below] at (\p,0) {$\scriptstyle{\w}$};   \node at (\p,0) {\t};  } \foreach \s/\r/\sa/\ea in {8/1/60/120,11/1/60/120}  {   \draw (\s,0.2) arc [radius=\r, start angle=\sa, end angle=\ea];  } \end{tikzpicture} \end{equation}Note
that $p+t=m-r$ and $q+s=n-r$. We set $x_{i}=e^{\varepsilon_{i}}$
and \textrm{$y_{j}=e^{-\delta_{j}}$,} and we let
\[
X=\left(x_{i}^{a_{j}}\right)_{\underset{1\le j\le m-r}{{\scriptscriptstyle 1\le i\le m}}},\quad Y=\left(y_{j}^{b_{i}}\right)_{\underset{1\le j\le n}{{\scriptscriptstyle 1\le i\le n-r}}},\quad Z=\left(\frac{x_{i}^{z}y_{j}^{z}}{1+\left(x_{i}y_{j}\right)^{-1}}\right)_{\underset{1\le j\le n}{{\scriptscriptstyle 1\le i\le m}}}.
\]
Then, $X$ encodes the typical $\bullet$-entries, $Y$ encodes the
typical $\times$-entries, and $Z$ encodes atypical entries.

$\ $
\begin{thm}
\label{thm:determinantal character formula} Let $L$ be a KW-module
with special arc diagram (\ref{special tame for det formula}). Then
one has
\begin{equation}
e^{\rho}R\cdot\mbox{ch\ }L=\left(-1\right)^{r(t+q)}\begin{vmatrix}X & Z\\
0 & Y
\end{vmatrix}.\label{eq:det formula}
\end{equation}

\end{thm}
$\ $
\begin{rem}
The data defining (\ref{eq:det formula}) can also be determined from
the standard arc diagram of a KW-module $L$. In particular, the exponents
$a_{j}$ are the typical $\bullet$-entries, the exponents $b_{i}$
are the typical $\times$-entries, the exponent $z$ is the maximal
atypical entry, $r$ is the number of arcs, $t$ is the number of
typical $\bullet$-nodes below an arc, and $q$ is the number of typical
$\times$-nodes below an arc. See Remark \ref{Rem: description of special tame}
to recall the relationship between the standard arc diagram and the
special arc diagram.\end{rem}
\begin{example}
If the special arc diagram of $L$ is 
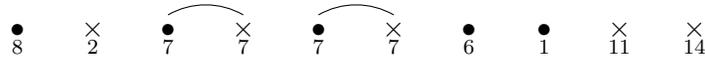
\begin{equation*}\begin{tikzpicture} 
%diagram
\node at (0,0) {$\bullet$};
\node [below] at (0,0) {$\scriptstyle 8$};
\node  at (1,0) {$\times$};
\node [below] at (1,0) {$\scriptstyle 2$};
\node at (2,0) {$\bullet$};
\node [below] at (2,0) {$\scriptstyle 7$};
\node  at (3,0) {$\times$};
\node [below] at (3,0) {$\scriptstyle 7$};
\node at (4,0) {$\bullet$};
\node [below] at (4,0) {$\scriptstyle 7$};
\node  at (5,0) {$\times$};
\node [below] at (5,0) {$\scriptstyle 7$};
\node at (6,0) {$\bullet$};
\node [below] at (6,0) {$\scriptstyle 6$};
\node at (7,0) {$\bullet$};
\node[below] at (7,0) {$\scriptstyle 1$};
\node  at (8,0) {$\times$};
\node[below] at (8,0) {$\scriptstyle 11$};
\node  at (9,0) {$\times$};
\node[below] at (9,0) {$\scriptstyle 14$};

%arcs
\draw (2,0.2) arc [radius=1, start angle=120, end angle= 60];
\draw (4,0.2) arc [radius=1, start angle=120, end angle= 60];

\end{tikzpicture}\end{equation*} then 
\[
e^{\rho}R\cdot\mbox{ch\ }L=-\left|\begin{array}{cccccccc}
{\scriptstyle x_{1}^{8}} & {\scriptstyle x_{1}^{6}} & {\scriptstyle x_{1}} & {\scriptstyle \frac{x_{1}^{{\scriptscriptstyle 7}}y_{1}^{{\scriptscriptstyle 7}}}{1+\left(x_{{\scriptscriptstyle 1}}y_{{\scriptscriptstyle 1}}\right)^{{\scriptscriptstyle -1}}}} & \frac{x_{1}^{7}y_{2}^{7}}{1+\left(x_{1}y_{2}\right)^{-1}} & \frac{x_{1}^{7}y_{3}^{7}}{1+\left(x_{1}y_{3}\right)^{-1}} & \frac{x_{1}^{7}y_{4}^{7}}{1+\left(x_{1}y_{4}\right)^{-1}} & \frac{x_{1}^{7}y_{5}^{7}}{1+\left(x_{1}y_{5}\right)^{-1}}\\
{\scriptstyle x_{2}^{8}} & {\scriptstyle x_{2}^{6}} & {\scriptstyle x_{2}} & {\scriptstyle \frac{x_{2}^{7}y_{1}^{7}}{1+\left(x_{2}y_{1}\right)^{-1}}} & \frac{x_{2}^{7}y_{2}^{7}}{1+\left(x_{2}y_{2}\right)^{-1}} & \frac{x_{2}^{7}y_{3}^{7}}{1+\left(x_{2}y_{3}\right)^{-1}} & \frac{x_{2}^{7}y_{4}^{7}}{1+\left(x_{2}y_{4}\right)^{-1}} & \frac{x_{2}^{7}y_{5}^{7}}{1+\left(x_{2}y_{5}\right)^{-1}}\\
{\scriptstyle x_{3}^{8}} & {\scriptstyle x_{3}^{6}} & {\scriptstyle x_{3}} & {\scriptstyle \frac{x_{3}^{7}y_{1}^{7}}{1+\left(x_{3}y_{1}\right)^{-1}}} & \frac{x_{3}^{7}y_{2}^{7}}{1+\left(x_{3}y_{2}\right)^{-1}} & \frac{x_{3}^{7}y_{3}^{7}}{1+\left(x_{3}y_{3}\right)^{-1}} & \frac{x_{3}^{7}y_{4}^{7}}{1+\left(x_{3}y_{4}\right)^{-1}} & \frac{x_{3}^{7}y_{5}^{7}}{1+\left(x_{3}y_{5}\right)^{-1}}\\
{\scriptstyle x_{4}^{8}} & {\scriptstyle x_{4}^{6}} & {\scriptstyle x_{4}} & \frac{x_{4}^{7}y_{1}^{7}}{1+\left(x_{4}y_{1}\right)^{-1}} & \frac{x_{4}^{7}y_{2}^{7}}{1+\left(x_{4}y_{2}\right)^{-1}} & \frac{x_{4}^{7}y_{3}^{7}}{1+\left(x_{4}y_{3}\right)^{-1}} & \frac{x_{4}^{7}y_{4}^{7}}{1+\left(x_{4}y_{4}\right)^{-1}} & \frac{x_{4}^{7}y_{5}^{7}}{1+\left(x_{4}y_{5}\right)^{-1}}\\
{\scriptstyle x_{5}^{8}} & {\scriptstyle x_{5}^{6}} & {\scriptstyle {\scriptstyle x_{5}}} & \frac{x_{1}^{7}y_{1}^{7}}{1+\left(x_{5}y_{1}\right)^{-1}} & \frac{x_{1}^{7}y_{2}^{7}}{1+\left(x_{5}y_{2}\right)^{-1}} & \frac{x_{5}^{7}y_{3}^{7}}{1+\left(x_{5}y_{3}\right)^{-1}} & \frac{x_{5}^{7}y_{4}^{7}}{1+\left(x_{5}y_{4}\right)^{-1}} & \frac{x_{6}^{7}y_{5}^{7}}{1+\left(x_{6}y_{5}\right)^{-1}}\\
{\scriptstyle 0} & {\scriptstyle 0} & {\scriptstyle 0} & {\scriptstyle y_{1}^{2}} & {\scriptstyle y_{2}^{2}} & {\scriptstyle y_{3}^{2}} & {\scriptstyle y_{4}^{2}} & {\scriptstyle y_{5}^{2}}\\
{\scriptstyle 0} & {\scriptstyle 0} & {\scriptstyle 0} & {\scriptstyle y_{1}^{11}} & {\scriptstyle y_{2}^{11}} & {\scriptstyle y_{3}^{11}} & {\scriptstyle {\scriptstyle y_{4}^{11}}} & {\scriptstyle y_{5}^{11}}\\
{\scriptstyle 0} & {\scriptstyle 0} & {\scriptstyle 0} & {\scriptstyle y_{1}^{14}} & {\scriptstyle y_{2}^{14}} & {\scriptstyle y_{3}^{14}} & {\scriptstyle y_{4}^{14}} & {\scriptstyle y_{5}^{14}}
\end{array}\right|.
\]

\end{example}

\subsection{Two linear-algebraic lemmas.}

For the proof of Theorem \ref{thm:determinantal character formula},
we will need the following lemmas. Suppose $m,n,r\in\mathbb{N}$ and
$m,n\geqslant r$, and let $\mathcal{M}$ be the set of matrices of
the form

\begin{equation}
\left(\begin{array}{ccc|ccc|ccc}
z_{11} & \cdots & z_{1,m-r} & z_{1,m-r+1} & \cdots & z_{1,m} & z_{1,m+1} & \cdots & z_{1,m+n-r}\\
\vdots & \ddots & \vdots & \vdots & \ddots & \vdots & \vdots & \ddots & \vdots\\
z_{m-r,1} & \cdots & z_{m-r,m-r} & z_{m-r,m-r+1} & \cdots & z_{m-r,m} & z_{m-r,m+1} & \cdots & z_{m-r,m+n-r}\\
\hline z_{m-r+1,1} & \cdots & z_{m-r+1,m-r} & z_{m+1-r,m-r+1} & \cdots & z_{m-r+1,m} & z_{m-r+1,m+1} & \cdots & z_{m-r+1,m+n-r}\\
\vdots & \ddots & \vdots & \vdots & \ddots & \vdots & \vdots & \ddots & \vdots\\
z_{m,1} & \cdots & z_{m,m-r} & z_{m,m-r+1} & \cdots & z_{m,m} & z_{m,m+1} & \cdots & z_{m,m+n-r}\\
\hline  &  &  & z_{m+1,m-r+1} & \cdots & z_{m+1,m} & z_{m+1,m+1} & \cdots & z_{m+1,m+n-r}\\
 & 0 &  & \vdots & \ddots & \vdots & \vdots & \ddots & \vdots\\
 &  &  & z_{m+n-r,m-r+1} & \cdots & z_{m+n-r,m} & z_{m+n-r,m+1} & \cdots & z_{m+n-r,m+n-r}
\end{array}\right).\label{eqn:matrix_form}
\end{equation}

\begin{lem}
\label{lem:conditions}There exists a unique polynomial $d(z_{ij})$,
$d:\mathcal{M\rightarrow\mathbb{C}}$, such that \end{lem}
\begin{enumerate}
\item $d$ is antisymmetric with respect to interchanges of the \emph{first}
$m$ rows and of the \emph{last} $n$ columns;
\item $d$ is linear with respect to row operations on the first $m$ rows
and column operations on the last $n$ columns; 
\item $d$ specializes to $1$ on all matrices of the form 
\[
\left(\begin{array}{c|c|c}
I_{m-r} & 0 & Z\\
\hline 0 & I_{r} & 0\\
\hline 0 & 0 & I_{n-r}
\end{array}\right),
\]
where $I$ denotes the identity matrix and $Z$ is arbitrary.\end{enumerate}
\begin{proof}
The determinant satisfies all of the required conditions, proving
existence. The proof of uniqueness is the same as the standard proof
of the uniqueness of the determinant. Using row and column operations
we can reduce any matrix in $\mathcal{M}$ to the form in condition
(3). \end{proof}
\begin{rem}
Observe that a permutation of the first $m$ rows of a matrix $M\in\mathcal{M}$
corresponds to a permutation of the subset $\{1,\dots,m\}$ of the
first indices of the elements $z_{i,j}$, and a permutation of the
last $n$ columns of a matrix $M\in\mathcal{M}$ corresponds to a
permutation of the subset $\{m+1-r,\dots,m+n-r\}$ of the second indices
of the elements $z_{i,j}$.\end{rem}
\begin{lem}
\label{lem:general_det_formula} Let $M\in\mathcal{M}$ be a matrix
of the form (\ref{eqn:matrix_form}). Then 
\[
\det\ M=\frac{1}{r!}\sum_{w\in Sym_{m}\times Sym_{n}}(-1)^{l(w)}w\left(z_{1,1}z_{2,2}\dots z_{m+n-r,m+n-r}\right),
\]
where $Sym_{m}$ permutes the subset $\{1,\dots,m\}$ of the first
indices of the elements $z_{i,j}$, and $Sym_{n}$ permutes the subset
$\{m+1-r,\dots,m+n-r\}$ of the second indices of the elements $z_{i,j}$.\end{lem}
\begin{proof}
By the Lemma \ref{lem:conditions}, we only need to check that the
polynomial on the right hand side satisfies the three conditions.
The first two are clear. For the third, we need to count the number
of elements $(u,v)\in Sym_{m}\times Sym_{n}$ that give a nonzero
contribution to the alternating sum. It is easily seen that $u|_{\{1,\dots,m-r\}}$
and $v|_{\{m+1,\dots,m+n-r\}}$ must be identities. So to get a nonzero
contribution (which must be 1) from the central block we need that
$u|_{\{m+1-r,\dots,m\}}=v|_{\{m+1-r,\dots,m\}}$. These elements are
all even, and there are $r!$ of them.
\end{proof}

\subsection{Proof of Theorem \ref{thm:determinantal character formula}.}

Let $L$ be a KW-module with special arc diagram (\ref{special tame for det formula}).
We define 
\begin{eqnarray*}
t_{\lambda} & = & \frac{e^{\lambda^{\rho}}}{\prod_{\beta\in S_{\lambda}}\left(1+e^{-\beta}\right)}\\
 & = & \prod_{i=1}^{p}x_{i}^{a_{i}}\prod_{i=p+1}^{m-r}x_{r+i}^{a_{i}}\prod_{k=1}^{r}\frac{\left(x_{p+k}y_{q+k}\right)^{z}}{1+\left(x_{p+k}y_{q+k}\right)^{-1}}\prod_{j=1}^{q}y_{j}^{b_{j}}\prod_{j=q+1}^{n-r}y_{r+j}^{b_{j}}.
\end{eqnarray*}
Then the Kac-Wakimoto character formula implies that 
\begin{equation}
e^{\rho}R\cdot\mbox{ch }L=\frac{1}{r!}\sum_{w\in Sym_{m}\times Sym_{n}}\left(-1\right)^{l\left(w\right)}w\left(t_{\lambda}\right),\label{eq:KW det intermediate}
\end{equation}
Now by applying Lemma \ref{lem:general_det_formula} to the matrix
in question we obtain

\[
\begin{array}{rcl}
{\displaystyle \begin{vmatrix}X & Z\\
0 & Y
\end{vmatrix}} & = & {\displaystyle \frac{1}{r!}\sum_{w\in Sym_{m}\times Sym_{n}}(-1)^{l(w)}w\left(\prod_{i=1}^{m-r}{x_{i}}^{a_{i}}\prod_{k=1}^{r}\frac{(x_{m-r+k}y_{k})^{z}}{1+(x_{m-r+k}y_{k})^{-1}}\prod_{j=1}^{n-r}{y_{r+j}}^{b_{r+j}}\right)}\\
 & = & {\displaystyle \frac{1}{r!}\sum_{w\in Sym_{m}\times Sym_{n}}(-1)^{l(w)}w\left((u^{-1},v^{-1}).t_{\lambda}\right),}
\end{array}
\]
where $u$ is the permutation sending $(x_{1},\dots,x_{m})$ to $(x_{1},\dots,x_{p},x_{p+r+1},\dots,x_{m},x_{p+1},\dots,x_{p+r})$,
and $v$ is the permutation sending $(y_{1},\dots,y_{n})$ to $(y_{q+1},\dots,y_{q+r},y_{1},\dots,y_{q},y_{q+r+1},\dots,y_{n})$.
Then since $l(u)=rt$ and $l(v)=rq$, we have

\begin{equation}
\begin{vmatrix}X & Z\\
0 & Y
\end{vmatrix}=\frac{(-1)^{r(t+q)}}{r!}\sum_{w\in Sym_{m}\times Sym_{n}}(-1)^{l(w)}w\left(t_{\lambda}\right).\label{eq:det of KW matrix}
\end{equation}
Combining (\ref{eq:KW det intermediate}) with (\ref{eq:det of KW matrix})
concludes the proof.

M.C.: Dept. of Mathematics, University of Michigan, Ann Arbor; mchmutov@umich.edu\\
C.H.: Dept. of Mathematics, Technion - Israel Institute of Technology; hoyt@tx.technion.ac.il\\ S.R.: Dept. of Mathematics, University of Michigan, Ann Arbor; shifi@umich.edu
\end{document}